\documentclass[11pt]{amsart}

\usepackage[inner=2.4cm,outer=2.4cm,top=2.4cm,bottom=2.4cm]{geometry}
\usepackage[english]{babel}
\usepackage{amsmath} 
\usepackage{amsthm,todonotes}
\usepackage{amsfonts}
\usepackage{amssymb}
\usepackage{graphicx}
\usepackage{mathrsfs}
\usepackage{color,eucal,enumerate,mathrsfs,mathtools}
\usepackage[normalem]{ulem}
\usepackage{amsmath,amssymb,epsfig,bbm}
\usepackage{pdfsync}
\usepackage{tikz-cd}
\numberwithin{equation}{section}
\usepackage[colorlinks,citecolor=green,linkcolor=red]{hyperref}
\newtheorem{claim}{\sc Claim}

\newcommand{\N}{\mathbb{N}}

\newcommand{\R}{\mathbb{R}}

\newcommand{\MM}{\mathscr{M}}

\newcommand{\PP}{\mathscr{P}}

\newcommand{\mm}{{\mbox{\boldmath$m$}}}

\newcommand{\cC}{{\mbox{\boldmath$C$}}}
\newcommand{\dD}{{\mbox{\boldmath$D$}}}

\newcommand{\sfd}{{\sf d}}
\newcommand{\sfe}{{\sf e}}

\newcommand{\sfC}{{\sf C}}

\newcommand{\restr}[1]{\lower3pt\hbox{$|_{#1}$}}

\newcommand{\down}{\downarrow}              
\newcommand{\up}{\uparrow}
\newcommand{\eps}{\varepsilon}  
\newcommand{\nchi}{{\raise.3ex\hbox{$\chi$}}}
\newcommand{\weakto}{\rightharpoonup}

\newtheorem{theorem}{Theorem}[section]

\newtheorem{corollary}[theorem]{Corollary}
\newtheorem{lemma}[theorem]{Lemma}
\newtheorem{proposition}[theorem]{Proposition}
\newtheorem{definition}[theorem]{Definition}

\newtheorem{remark}[theorem]{Remark}

\newcommand{\dist}{\mathrm{dist}}
\newcommand{\LIP}{\mathrm{LIP}}
\newcommand{\Lip}{\mathrm{Lip}}
\newcommand{\lip}{\mathrm{lip}}

\newcommand{\diam}{\mathrm{diam}}

\newcommand{\fr}{\hfill$\blacksquare$}   
\newcommand{\res}{\mathop{\hbox{\vrule height 7pt width .5pt depth 0pt
\vrule height .5pt width 6pt depth 0pt}}\nolimits} 

\newcommand{\CD}{{\sf CD}}
\newcommand{\mcp}{{\sf MCP}}
\renewcommand{\mm}{\mathfrak m}

\newcommand{\limi}{\varliminf}
\renewcommand{\limsup}{\varlimsup}
\renewcommand{\liminf}{\varliminf}
\renewcommand{\d}{{\rm d}}

\newcommand{\X}{{\rm X}}

\newcommand{\Z}{{\rm Z}}

\newcommand{\e}{\sfe}

\newcommand{\rmKe}{{\rm Ke}}
\newcommand{\rmCh}{{\rm Ch}}
\newcommand{\OptGeo}{{\rm OptGeo}}

\newcommand{\Xdm}{(\X,\sfd,\mm)}
\newcommand{\Xdmx}{(\X,\sfd,\mm,x)}
\newcommand{\Xdmxinf}{(\X_\infty,\sfd_\infty,\mm_\infty,x_\infty)}
\newcommand{\Xdmxn}{(\X_n,\sfd_n,\mm_n,x_n)}
\newcommand{\Comp}{{\rm Comp}}

\newcommand{\supp}{{\rm supp}}

\usepackage{ esint }
\usepackage{tcolorbox}
\title[]{Mosco-convergence of Cheeger energies on varying spaces satisfying curvature dimension conditions}

\author[]{Francesco Nobili} 
\address{Università di Napoli Federico II, Dipartimento di Matematica, Via Cintia, Monte S. Angelo, 80126 Napoli (NA), Italy}
\email{\url{francesco.nobili@unina.it}}

\author[]{Federico Renzi} 
\address{Scuola Normale Superiore, Piazza dei Cavalieri 7, 56126 Pisa, Italy}
\email{\url{federico.renzi@sns.it}}

\author[]{Federico Vitillaro} 
\address{Scuola Normale Superiore, Piazza dei Cavalieri 7, 56126 Pisa, Italy}
\email{\url{federico.vitillaro@sns.it}}

\begin{document}

\begin{abstract} 
We study the Mosco-convergence of Cheeger energies on Gromov-Hausdorff converging spaces satisfying different types of curvature dimension conditions. The case of functions of bounded variation is also considered, and applications to the continuity of Neumann eigenvalues are obtained. Our method, covering possibly infinite dimensional settings, is based on a Lagrangian approach and combines the stability properties of Wasserstein geodesics with the characterization of the nonsmooth calculus in duality with test plans.
\end{abstract}

\maketitle
\tableofcontents

\section{Introduction}
A remarkable feature of the synthetic theory of Ricci curvature lower bounds is its robust compactness and stability under Gromov–Hausdorff type convergences. In these settings, the curvature bounds remain stable because the zeroth-order convergence of the metric measure structure is complemented by a uniform second-order control. Once this framework is established, it is possible to investigate the stability of key quantities in geometric analysis such as heat flows \cite{Gigli10,GMS15}, heat kernels \cite{AmbrosioHondaTewodrose18}, regular Lagrangian flows \cite{AmbrosioStraTrevisan17}, as well as various functional constants and geometric stability properties, see e.g.\ \cite{Honda15,Honda17,AmbrosioHonda17,AmbrosioHondaPortegies18,HondaKettererMondelloPeralesRigoni24,NobiliViolo22,NobiliViolo24,NobiliViolo24_PZ,Nobili24_overview}. We refer also to \cite{Gigli23_working} for more details.

Starting from \cite{AmbrosioHonda17}, stability properties up to codimension one have been established, leading to significant advances in the understanding of the codimension-one structure of spaces with Ricci lower bounds \cite{AmbrosioBrueSemola19,BruPasSem19,BruPasSem21-constantcodimension}. These results have also proved effective in addressing isoperimetric problems on noncompact manifolds \cite{AntonelliBrueFogagnoloPozzetta22,AntonelliPasqualettoPozzettaSemola25}. We refer to \cite{Pozzetta23} for an overview. 

\medskip

In this manuscript, we study first-order stability properties of spaces satisfying the curvature dimension condition $\CD(K,N)$ of \cite{Lott-Villani09} and \cite{Sturm06I,Sturm06II} or the measure contraction property $\mcp(K,N)$ of \cite{Ohta07} and \cite{Sturm06II} both requiring, for some $K\in\R,N\in [1,\infty)$, a metric measure space $\Xdm$ to have Ricci curvature bounded below by $K$ and dimension bounded above by $N$, in a synthetic sense. The underlying relevant notion of convergence will be that of pointed measured Gromov convergence (pmG), under which these conditions are known to be stable \cite{Gromov07,Sturm06I,Sturm06II,Lott-Villani09,Ohta07,GMS15}. Here we adopt the so-called extrinsic approach \cite{GMS15}, which was studied in connection with previous notions of convergence.

Our main goal is to investigate the stability of the nonsmooth Sobolev and BV calculus along sequences of pmG-converging spaces. Recall that for a metric measure space $\Xdm$ and $f\in L^p(\mm)$, the membership $f\in W^{1,p}(\X)$ is characterized by the finiteness of the $p$-Cheeger energy $\rmCh_p(f)$, obtained via relaxation from Lipschitz functions. Similarly, for $f\in L^1(\mm)$, the property $f\in BV(\X)$ is defined by the finiteness of the total variation $|\dD f|(\X)$ (see Section \ref{sec:prelim} for details and references). Consider a sequence of pointed metric measure spaces $\Xdmxn$  that is pmG-converging to $\Xdmxinf$ (see Definition \ref{def:pmGH}), which we shortly write
\[
\X_n \overset{pmG}{\to} \X_\infty.
\]
When the spaces are uniformly bounded, we simply speak about measured Gromov convergence. For functions $f_n\in L^p(\mm_n)$ with $p\in[1,\infty)$, there are natural notions of $L^p$-weak convergence to a limit $f_\infty\in L^p(\mm_\infty)$ (see Definition \ref{def:Lp convergence pmGH}). A natural question is whether the following semicontinuity property holds:
\begin{equation}
\begin{array}{lll}
   f_n \to f_\infty \quad L^p\text{-weak}& &\sup_{n\in \N}\rmCh_p(f_n)<\infty\\
   f_n \to f_\infty \quad L^1\text{-weak} & &\sup_{n\in \N}|\dD f_n|(\X_n)<\infty
\end{array}\Big\}\overset{?}{\implies} \begin{array}{l}
f_\infty \in W^{1,p}(\X_\infty),\\
f_\infty \in BV(\X_\infty),
\end{array}\label{intro:semicontinuity}
\end{equation}
and when this is true in the sharp form
\begin{equation}
\rmCh_p(f_\infty)\le \liminf_{n\uparrow \infty}\rmCh_p(f_n),\qquad |\dD f_\infty|(\X_\infty)\le \liminf_{n\uparrow \infty}|\dD f_n|(\X_n).\label{intro:gammaliminf}
\end{equation}
The above property, when valid, fits in the general framework of the Mosco-convergence of Cheeger energies (see \cite{Gigli23_working}). Indeed, we recall that the existence of the so-called \emph{recovery sequences} is always true without further assumptions, namely for every $f_\infty\in L^p(\mm_\infty)$ there are $f_n \in L^p(\mm_n)$ which are $L^p$-strong converging to $f_\infty$ and satisfies (depending on $p>1$ or $p=1$):
\[
    \limsup_{n\uparrow \infty}\rmCh_p(f_n) \le \rmCh_p(f_\infty),\qquad \limsup_{n\uparrow \infty}|\dD f_n|(\X_n)\le  |\dD f_\infty|(\X_\infty).
\]
To the best of our knowledge, this fact was explicitly pointed out for the first time in \cite[Theorem 4.4]{Gigli23_working}, whose proof depends on the upper semicontinuity of the asymptotic Lipschitz constant (see also the proof of \cite[Theorem 8.1]{AmbrosioHonda17}, as well as the general analysis in \cite{AmbrosioGigliSavare11-3}). Given the importance of this result for our goals, we single out a statement in Theorem \ref{th:limsup}.

On the contrary, even though \eqref{intro:gammaliminf} is always true when the underlying space is fixed, in this varying base-space setting it might very well be false.  Notable counterexamples are given by passage to the limit from discrete to continuum, while homogeneization results also provide standard constructions even with uniform doubling and Poincar\'e conditions. In these cases, the issue is that we are asking for a first-order stability to hold under a zeroth-order convergence.

\medskip
The main goal of this note is the following:
\begin{center}
to study \eqref{intro:semicontinuity},\eqref{intro:gammaliminf} assuming the $\CD(K,N)$ or $\mcp(K,N)$ condition. 
\end{center}
Notice that this investigation is at least reasonable, as we are enforcing a uniform second-order curvature bound along a sequence of spaces to investigate a first-order stability.

\medskip

\noindent\textbf{Statement of the main results}.
We now present our main results.  For technical reasons appearing in different parts of this note, we also require the essentially non-branching property (see Section \ref{sec:prelim}). 
\begin{theorem}\label{thm:Mosco in CDKN}
   Let $\Xdmxn$ with $\mm_n(\X_n)<\infty$ be $q$-essentially non-branching pointed metric measure spaces for all $q \in (1,\infty),n\in\N$. Suppose that $\X_n$ satisfies ${\sf  CD}(K,N)$ for some $K\in\R,N\in[1,\infty)$ and that $\X_n \overset{pmG}{\rightarrow} \X_\infty$ for some $\Xdmxinf$. Then, it holds
\begin{itemize}
\item[(i)] for all $p\in(1,\infty)$, if $f_n\in L^p(\mm_n)$ converges $L^p$-weak to $f_\infty \in L^p(\mm_\infty)$ we have
\[
 \rmCh_p(f_\infty)\le \liminf_{n\to\infty}\rmCh_p(f_n);
\]
\item[(ii)] if $f_n\in L^1(\mm_n)$ converges $L^1$-weak to $f_\infty \in L^1(\mm_\infty)$, we have
\[
 |\dD f_\infty|(\X_\infty)\le \liminf_{n\to\infty}|\dD f_n|(\X_n).
\]
\end{itemize} 
\end{theorem}
Except when $p = 2$, obtained in \cite{GMS15}, the above result is new for all other exponents. For its proof, a major role is played by the result \cite{ACCMcS20} about the independence of $\CD_q$ condition of the transport exponent $q\in(1,\infty)$ (reported in Theorem \ref{thm:CDq independent} below). This makes it possible to treat all the $p$-Cheeger and BV-energies together, only assuming working with the $\CD = \CD_2$ condition. As a byproduct, we also have to assume the finiteness of the reference measures $\mm_n$ since this is an assumption in \cite{ACCMcS20}. Our statement would automatically extend to general $\sigma$-finite reference measures as soon as \cite{ACCMcS20} works in this case.  

In contrast, $\mm_\infty$ can be also $\sigma$-finite, therefore the above result can be certainly applied for non-compact limit spaces $\X_\infty$ arising from sequences of compact ones (for instance, as blow-ups). Furthermore, conclusion (i) is directly implied by an analogous result in  Theorem \ref{thm:Mosco in CDKinf} holding in a genuine infinite dimensional setting, up to technical assumptions fulfilled in Theorem \ref{thm:Mosco in CDKN} thanks to \cite{ACCMcS20}.

\medskip 
We next turn to the analogous result for spaces satisfying the measure contraction property.
\begin{theorem}\label{thm:Mosco in MCP}
Let $\Xdmxn$ be a sequence of $q$-essentially non-branching pointed metric measure spaces for all $q\in(1,\infty),n\in\N$. Suppose $\X_n$ satisfies ${\sf MCP}(K,N)$ for some $K \in \R, N \in [1,\infty)$ and that $\X_n \overset{pmG}{\rightarrow} \X_\infty$ for some $\Xdmxinf$. Then, it holds
\begin{itemize}
\item[(i)] for all $p\in(1,\infty)$, if $f_n\in L^p(\mm_n)$ converges $L^p$-weak to $f_\infty \in L^p(\mm_\infty)$ we have
\[
 \rmCh_p(f_\infty)\le 2^N\liminf_{n\to\infty}  \rmCh_p(f_n);
\]
\item[(ii)] if $f_n\in L^1(\mm_n)$ converges $L^1$-weak to $f_\infty \in L^1(\mm_\infty)$, we have
\[
 |\dD f_\infty|(\X_\infty)\le 2^N\liminf_{n\to\infty}|\dD f_n|(\X_n).
\]
\end{itemize}
\end{theorem}
The above is the first stability results for Cheeger energies along converging spaces satisfying uniform measure contraction properties. Unlike Theorem \ref{thm:Mosco in CDKN}, here we do not require the reference measures to be finite. In particular, the result also applies to ${\sf CD}(K,N)$ spaces, since these satisfy ${\sf MCP}(K,N)$. The reason is that the measure contraction property is already independent of the transport exponent (see \cite[Remark 5.2]{GigliNobili22}). This completely bypasses the use of \cite{ACCMcS20} and allows simultaneous treatment of all $p$-Cheeger and BV energies. The drawback is the appearance of a multiplicative factor $2^N$, that is we still obtain \eqref{intro:semicontinuity}, though not the sharp form \eqref{intro:gammaliminf}. This loss stems from the weaker interpolation estimates available under ${\sf MCP}$ (see \cite{CavMon17} and Section \ref{sec:mcp}).

\begin{remark}\label{rem:branch}
    \rm We point out that in both Theorem \ref{thm:Mosco in CDKN} and Theorem \ref{thm:Mosco in MCP}, the limit space $(\X_\infty,\sfd_\infty,\mm_\infty)$ can be also branching. In view of applications (see for instance Theorem \ref{thm:neumann}), this is highly desirable as the essential non-branching property is known to be not stable within the ${\sf CD}$ or ${\sf MCP}$ classes.
    \fr
\end{remark}

\medskip 
 It is well known that Mosco-convergence results for Cheeger energies and Dirichlet forms have deep implications in the analysis and geometry of spaces with Ricci curvature lower bounds. We refer to the aforementioned references and the upcoming paragraphs for some relevant literature. As a byproduct of our results, we obtain the continuity of Neumann eigenvalues in the measured Gromov convergence. Recall that for every $p \in (1,\infty)$, the first (nontrivial) Neumann eigenvalue of the $p$-Laplacian in a metric measure space $\Xdm$ is defined as the nonnegative quantity
\[
    \lambda_p(\X) \coloneqq \inf \left\{ \frac{\rmCh_p(f)}{\|f\|^p_{L^p(\mm)}} \colon f \in W^{1,p}(\X),\,\int f|f|^{p-2}\, \d \mm =0\right\}.
\]
\begin{theorem}\label{thm:neumann}
   Let $(\X_n,\sfd_n,\mm_n)$ be $q$-essentially non-branching metric measure spaces satisfying ${\sf CD}(K,N)$ for some $K \in \R,N\ge 1$ and for all $q\in(1,\infty),n\in\N$. Suppose that $\sup_{n\in\N} {\rm diam}(\X_n)<\infty$ and that $\X_n \overset{pmG}{\to}\X_\infty$ for some  $(\X_\infty,\sfd_\infty,\mm_\infty)$. Then, for all $p\in(1,\infty)$, it holds
    \begin{equation}
        \lambda_p(\X_\infty) = \lim_{n\to \infty} \lambda_p(\X_n).
    \label{eq:neumann continuous}
    \end{equation}
\end{theorem}
The above for $p=2$ is implied by the analysis in \cite{GMS15}, and for general $p\in(1,\infty)$ by that in \cite{AmbrosioHonda17} but only in the restricted class of ${\sf RCD}$-spaces (see below for bibliographic references).  Therefore, in the non-branching ${\sf CD}$ class our result is new for all $p\neq 2$. We also note that, when $K>0$, the assumption on the diameters is automatically satisfied by the maximal diameter theorem (\cite{Sturm06II}).

\medskip

\noindent\textbf{Strategy of the proof}. 
The main proof-argument combines a Lagrangian characterization of Sobolev and BV functions, in the spirit of
\cite{AmbrosioGigliSavare11,AmbrosioGigliSavare11-3,AmbrosioDiMarino14},
with polygonal interpolation arguments similar to \cite{Lisini07}, suitably adapted to preserve precise density estimates. To illustrate our methods, we consider  the $2$-Cheeger energy and the ${\sf CD}(K,\infty)$ setting with $K\ge 0$ as it contains all the key ideas. In particular, we show that infinite dimensional settings are covered as well. Afterwards, we comment on the general case $K\in\R$, which is substantially more involved.
\medskip 

The first ingredient, that will be proved in Theorem \ref{thm:Sobolevintegrated}, is that on an arbitrary metric measure space $\Xdm$ and $f \in L^2(\mm)$, we have $f\in W^{1,2}(\X)$ if and only if there is $C\ge 0$ so that
\begin{equation}
\int f(\gamma_1)-f(\gamma_0)\,\d \eta \le \Comp(\eta)^{1/2}\rmKe_2^{1/2}(\eta)C,
\label{intro:claim}
\end{equation}
for all test plans $\eta \in \PP(C([0,1],\X))$ (see Section \ref{sec:prelim} for the related notation). In this case, we have
\[
\rmCh^{1/2}_2(f) = \min \{C  \colon \eqref{intro:claim} \text{ holds with }C\ge 0\}.
\]
Consider now a sequence $\Xdmxn$ of pointed $\CD(0,\infty)$ spaces pmG-converging to some limit $\Xdmxinf$. Fix also $f_n \in L^2(\mm_n)$ that is $L^2$-weak converging to some $f_\infty \in L^2(\mm_\infty)$.  For all test plans $\eta$ with bounded support (without loss of generality, see Remark \ref{rmk:planboundedW1p}), it is possible to find optimal dynamical plans $\pi_n \in \PP(C([0,1],\X_n))$ satisfying
\begin{equation}
\limsup_{n\uparrow\infty}\Comp(\pi_n)\le \Comp(\eta), \qquad \limsup_{n\uparrow\infty }\rmKe_2(\pi_n)\le \rmKe_2(\eta),
\label{intro:CompKe principles}
\end{equation}
while satisfying the end-point convergence
\begin{equation}\label{intro:endpoints}
\frac{\d (\e_0)_\sharp \pi_n}{\d \mm_n} \to  \frac{\d (\e_0)_\sharp \eta}{\d \mm_\infty},\qquad 
     \frac{\d (\e_1)_\sharp \pi_n}{\d \mm_n} \to \frac{\d (\e_1)_\sharp \eta}{\d \mm_\infty},\qquad \text{in $L^2$-strong as $n\uparrow\infty$}.
\end{equation}
In the first of \eqref{intro:CompKe principles}, it is crucial that $K\ge 0$, see Proposition \ref{prop:q_polygonal}. Since, eventually, $f_n \in W^{1,2}(\X_n)$, we have for all $n$ big enough  
\[
\int f_n(\gamma_1)-f_n(\gamma_0)\,\d\pi_n \overset{\eqref{intro:claim}}{\le} \Comp(\pi_n)^{1/2}\rmKe_2^{1/2}(\pi_n)\rmCh_2^{1/2}(f_n).
\]
Taking $n\uparrow\infty$ and using the properties \eqref{intro:CompKe principles},\eqref{intro:endpoints}, the above implies, using again \eqref{intro:claim}, that
\[
f_\infty \in W^{1,2}(\X_\infty)\qquad \text{and}\qquad \rmCh_2(f_\infty) \le \liminf_{n\uparrow\infty}\rmCh_2(f_n).
\]

The case when $K < 0$ is instead much harder and presents some technical difficulties to be dealt with. In this case, a polygonal geodesic plan rather than a single dynamical plan must be built and shown to satisfy suitable modifications of \eqref{intro:CompKe principles} and \eqref{intro:endpoints}. This set of ideas was already faced in \cite{GigliNobili22,NobiliPasqualettoSchultz22} in the fixed base-space setting and strongly inspired by \cite{Lisini07}. However, we point out that our construction presents some novelties and new difficulties that are faced for the first time. The main challenge we face is the need to discard mass along the geodesic interpolations by restricting to properly chosen geodesics with maximal length under control. Unfortunately, this procedure does not yet guarantee the validity of \eqref{intro:endpoints} and, therefore, we need to simultaneously refine the polygonal construction by increasing the number of interpolants. In this way, the key properties \eqref{intro:CompKe principles},\eqref{intro:endpoints} will be asymptotically recovered (see Proposition \ref{prop:q_polygonal}, Proposition \ref{prop:infty_polgeo} for the polygonal constructions and Remark \ref{rem:polgeo vs old polgeo} for comparison with previous works).  

\medskip

\noindent\textbf{Comparison with previous literature}. One of the first motivations to investigate the kind of stability results of this note came from the celebrated theory of Ricci limit spaces \cite{Cheeger-Colding96}. Indeed, in \cite{Cheeger-Colding97III} the stability of the Laplacian spectrum was studied along convergence of manifolds with Ricci curvature lower bounds, solving a conjecture in \cite{Fukaya87}. 

In the settings and with the formalism adopted in this note, the Mosco-convergence of the $2$-Cheeger energies on varying $\CD(K,\infty)$ spaces was first faced in \cite{GMS15} (also motivated by \cite{Cheeger-Colding97III}). We point out that a key point is played by the identification proved in \cite{AmbrosioGigliSavare11} between the Fisher-information, i.e.\ the squared negative slope of the Entropy functional (see \eqref{eq:Entm}), and the $2$-Cheeger energy:
\begin{equation}
|\partial^- {\rm Ent}_\mm|^2(\rho\mm) = 4\rmCh_2(\sqrt{\rho}). \label{eq:Fisher Cheeger}
\end{equation}
This transfers information from a Lagrangian level to an Eulerian Level: the Mosco-convergence of the $2$-Cheeger energy is indeed implicitly achieved in \cite[Theorem 6.8]{GMS15} as a by-product of the analogous stability results for the Fisher information.  The relevant notion of convergence, and the stability of the heat flow, were originally clarified in \cite{Gigli10}.

Later, in \cite{AmbrosioHonda17}, the results of \cite{GMS15} were extended in the form of Gamma-convergence, and of Mosco-convergence under additional hypotheses (recently removed in \cite{NobiliViolo25} for finite dimensional settings) to the whole range of $p$-Cheeger energies and up to the case of BV-functions. However, in \cite{AmbrosioHonda17} it is crucial to work in the Riemannian sub-class of the so-called ${\sf RCD}(K,\infty)$ spaces \cite{AmbrosioGigliSavare11-2} (see \cite{AmbICM} and references therein). This was possible by exploiting the self-improvement properties of heat flow \cite{Savare13}.

More recently, in \cite{CarronMondelloTewodrose21,CarronMondelloTewodrose22} a Mosco-convergence type result was deduced for sequences of manifolds with Kato bounds on the Ricci curvature. This was achieved with a set of different techniques relying on stability properties of Dirichlet spaces (\cite{Kasue05,KuwaeShioya03}).

In this note, we follow a direct Lagrangian approach, avoiding the use of \eqref{eq:Fisher Cheeger} and, more generally, of any Eulerian viewpoint in the main proof-argument until the point where the Lagrangian formulation of the nonsmooth calculus is proved to be equivalent to a relaxed notion (cf.\ \eqref{intro:claim}). For this equivalence, the bridge between the Lagrangian and Eulerian viewpoint is naturally played by the so-called Kuwada's Lemma (see \cite{AmbrosioGigliSavare11-3}).

Since the curvature dimension conditions are formulated in terms of the convexity of entropies of Wasserstein geodesics, our approach has the advantage of being direct and conceptually simple. Thanks to our methods, we are able to achieve Mosco-convergence stability results covering a wide class of possible non-Riemannian spaces. Moreover, our strategy has the potential to deal with local curvature constraints, as it ultimately looks at interpolations in bounded regions. On this note, we show in Proposition \ref{prop:liminf on open} and Proposition \ref{prop:liming BV on open} that it can be localized on open subsets, while in the forthcoming work \cite{AmbrosioNobiliRenziVitillaro26} we shall rigorously verify that uniformly local Riemannian Ricci curvature lower bounds are stable. Finally, as finite-dimensional ${\sf RCD}$-spaces are non-branching (\cite{Deng25}), all our results also apply in this case. 

\section{Preliminaries}\label{sec:prelim} 

\subsection{Test plans}
A metric measure space is a triple \((\X,\sfd,\mm)\), where
\[\begin{split}
(\X,\sfd)&\quad\text{ is a complete, separable metric space},\\
\mm\geq 0&\quad\text{ is a boundedly-finite Borel measure on }(\X,\sfd).
\end{split}\]

We denote by \(\MM_+(\X),\mathscr P(\X)\) respectively the family of Borel nonnegative measures, and Borel probability
measures on \((\X,\sfd)\) equipped with the weak topology in duality with continuous and bounded functions \(C_b(\X)\), which we shall often write $\mu_n \weakto \mu$ for $\mu_n,\mu \in \PP(\X)$. Given $\varphi \colon \X \to {\rm Y}$, for a metric target ${\rm Y}$, and $\mu \in \PP(\X)$, we define $\varphi_\sharp \mu (B) \coloneqq \mu(\varphi^{-1}(B))$, for all $B\subset {\rm Y}$ Borel.

Denote $\mathcal L^1$ the one-dimensional Lebesgue measure restricted to $[0,1]$. Denote by $C([0,1],\X)$ the set of continuous and $\X$-valued curves equipped with the sup distance. For \(q\in[1,\infty]\), denote by
\(AC^q([0,1],\X)\) the family of continuous curves
\(\gamma\in C([0,1],\X)\) so that there is \(g\in L^q(0,1)\) satisfying
\begin{equation}\label{eq:def_AC_curve}
\sfd(\gamma_t,\gamma_s)\leq\int_s^t g(r)\,\d r,
\qquad \forall s,t\in[0,1]\text{ with }s\leq t.
\end{equation}
We simply denote $AC([0,1],\X)\coloneqq AC^1([0,1],\X)$ and ${\rm LIP}([0,1],\X)\coloneqq AC^\infty([0,1],\X)$. It is well known that, given \( \gamma\in AC^q([0,1],\X)\) there exists
\(|\dot\gamma_t|\coloneqq\lim_{h\to 0}\sfd(\gamma_{t+h},\gamma_t)/|h|\) for a.e.\ \(t\in[0,1]\). It holds that  $|\dot\gamma| \in L^q(0,1)$ and it can be equivalently characterized as the minimal a.e.\ \(g\in L^q(0,1)\) satisfying \eqref{eq:def_AC_curve} (see \cite[Theorem 1.1.2]{AmbrosioGigliSavare08}).
It will be convenient to define the metric speed functional
\({\sf ms}\colon C([0,1],\X)\times[0,1]\to[0,+\infty]\) as follows:
\[
{\sf ms}(\gamma,t)\coloneqq|\dot\gamma_t|,
\quad\text{ whenever }\gamma\in AC([0,1],\X)\text{ and }
\exists\lim_{h\to 0}\frac{\sfd(\gamma_{t+h},\gamma_t)}{|h|},
\]
and \({\sf ms}(\gamma,t)\coloneqq+\infty\) otherwise.
It holds that \(\sf ms\) is a Borel function (see, e.g., \cite{GP19}). We will often consider probability measures $\pi \in \PP(C([0,1],\X))$ and call any such $\pi$ bounded or of bounded support, provided $\{\gamma_t \colon \gamma \in \supp(\pi),t \in [0,1]\}\subset \X$ is bounded. 

We recall the notion of test plan \cite{AmbrosioGigliSavare11,AmbrosioGigliSavare11-3}.
\begin{definition}[$q$-test plan]\label{def:test_plan}
Let \((\X,\sfd,\mm)\) be a metric measure space and \(q\in[1,\infty]\). A measure \(\pi\in\mathscr P\big(C([0,1],\X)\big)\) is a $q$-test plan, provided
\begin{subequations}\begin{align}\label{eq:def_test_plan1}
&\exists\,{\rm C}>0:\quad\forall t\in[0,1],\quad(\e_t)_\sharp\pi\leq{\rm C}\mm,\\
\label{eq:def_test_plan2}
&\|{\sf ms}\|_{L^q(\pi\otimes\mathcal L^1)}<+\infty.
\end{align}\end{subequations}
The minimal constant $C\ge 0$ for \eqref{eq:def_test_plan1} to hold is called the compression constant and denoted by \({\rm Comp}(\pi)\). If  $q\in(1,\infty)$, the kinetic energy of a $q$-test plan $\pi$ is defined by
\[
{\rm Ke}_q(\pi)\coloneqq \|{\sf ms}\|_{L^q(\pi\otimes\mathcal L^1)}^q.
\]
Similarly, the Lipschitz constant of an $\infty$-test plan $\pi$ is defined by 
\[
    \Lip(\pi)\coloneqq \|{\sf ms}\|_{L^\infty(\pi\otimes\mathcal L^1)}.
\]
\end{definition}

Notice that any \(q\)-test plan must be concentrated on \(AC^q([0,1],\X)\). The kinetic \(q\)-energy can be extended to a functional
\({\rm Ke}_q\colon\mathscr P\big(C([0,1],\X)\big)\to[0,+\infty]\)
by declaring that \({\rm Ke}_q(\pi)\coloneqq+\infty\) whenever \(\pi\)
is not a \(q\)-test plan. Similarly, for $q=\infty$, we extend $\Lip \colon\mathscr P\big(C([0,1],\X)\big)\to[0,+\infty]$ by declaring that $\Lip(\pi)=\infty$ if \(\pi\)
is not an \(\infty\)-test plan. Notice that $\Lip(\pi)<\infty$ implies that $\pi$ is concentrated on $\Lip([0,1],\X)$ and \({\rm Lip}(\pi)\) can be equivalently characterized as the minimal \({\rm L}\geq 0\) such that \(\pi\) is concentrated on \({\rm L}\)-Lipschitz curves (see \cite[Remark 2.2]{NobiliPasqualettoSchultz22}).

We recall next the following semicontinuity result from \cite[Proposition 2.4]{NobiliPasqualettoSchultz22} (notice that the reference measure is never needed as an assumption there): given $q_k\uparrow\infty$ and \((\pi_k)\subset
\mathscr P\big(C([0,1],\X)\big)\) with \(\pi_k\rightharpoonup\pi\)
as \(k\uparrow\infty\) in duality with $C_b(\X)$, then
\begin{equation}\label{eq:Mosco Keq Lip}
{\rm Lip}(\pi)\leq\limi_{k\to\infty}{\rm Ke}_{q_k}^{1/q_k}(\pi_k).
\end{equation}

We recall the definition of evaluation map at time $t \in [0,1]$: given $\gamma \in C([0,1],\X)$, we define $\e_t(\gamma)\coloneqq \gamma_t$. Notice that $\e_t \colon  C([0,1],\X) \to \X$ is continuous. We conclude this part by recalling a compactness result for $\infty$-test plans we are going to use in this note. See \cite[Proposition 2.6]{NobiliPasqualettoSchultz22}.
\begin{proposition}\label{prop:infty_tightness}
Let \((\X,\sfd,\mm)\) be a metric measure space. Let \((\pi_n)_n\)
be a sequence of \(\infty\)-test plans such that \(\bigcup_{n}{\rm spt}\big((\e_0)_\sharp\pi_n\big)\)
is bounded, \(\sup_{n}{\rm Lip}(\pi_n)<+\infty\),
and \(\sup_{n}{\rm Comp}(\pi_n)<+\infty\).
Then there exist a subsequence $n_k\uparrow\infty$ and a \(\infty\)-test plan \(\pi\) such that \(\pi_{n_k}\rightharpoonup\pi\)
as \(k\uparrow\infty\).
\end{proposition}

\subsection{Wasserstein space}
We recall some basic facts around Optimal Transport on a metric space $(\X,\sfd)$. We refer to \cite{Villani09,AmbrosioBrueSemola24_Book} and references therein for a complete account of the theory.

A measure \(\alpha\in\mathscr \MM_+(\X\times\X)\)
is called admissible plan between two measures \(\mu_0,\mu_1 \in \MM_+(\X)\) with the same mass provided
\(P^1_\sharp\alpha=\mu_0\) and \(P^2_\sharp\alpha=\mu_1\), where we denoted
\(P^1,P^2\colon\X\times\X\to\X\) the projection maps onto the first and second components, respectively. We denote by \({\rm Adm}(\mu_0,\mu_1)\) the family of admissible plans between $\mu_0$ and $\mu_1$. For $q<\infty$, denote $\PP_q(\X) \subset \PP(\X)$ the family of probability measures with finite $q$-moment and define the \(q\)-Wasserstein distance 
\begin{equation}\label{eq:def_Wq}
W_q^q(\mu_0,\mu_1)\coloneqq\inf_{\alpha\in{\rm Adm}(\mu_0,\mu_1)}
\int\sfd^q(x,y)\,\d\alpha(x,y)
,\qquad \forall
\mu_0,\mu_1\in\mathscr P_q(\X).
\end{equation}
In the limit case \(q=\infty\), we denote by $\PP_\infty(\X)\subset \PP(\X)$ the family of boundedly supported probability measures and define the \(\infty\)-Wasserstein distance
\(W_\infty\)  as
\begin{equation}\label{eq:def_Winfty}
W_\infty(\mu_0,\mu_1)\coloneqq\inf_{\alpha\in{\rm Adm}(\mu_0,\mu_1)}\, \|\sfd(\cdot,\cdot)\|_{L^\infty(\alpha)},\qquad \forall \mu_0,\mu_1\in\mathscr P_\infty(\X).
\end{equation}
It is well-known (see, e.g., \cite{GivSho84,ChaDeP08}) that $W_\infty$ is the monotone limit of the $q$-Wasserstein distance:
\begin{equation}\label{eq:lim_Wq}
W_\infty(\mu_0,\mu_1)=\lim_{q\to\infty}W_q(\mu_0,\mu_1)= \sup_{q>1}W_q(\mu_0,\mu_1),
\qquad \forall \mu_0,\mu_1\in\mathscr P_\infty(\X) .
\end{equation}
We denote by \({\rm Opt}_q(\mu_0,\mu_1)\) the set of all optimal
plans between \(\mu_0\) and \(\mu_1\), i.e.\ of all minimisers
of \eqref{eq:def_Wq} or \eqref{eq:def_Winfty}. The couple $(\mathscr P_q(\X),W_q)$ is called the $q$-Wasserstein space.

We next introduce the so-called optimal dynamical plans following \cite{AmbrosioBrueSemola24_Book}. A geodesic in \((\X,\sfd)\) is a curve \(\gamma\in C([0,1],\X)\) satisfying \(\sfd(\gamma_t,\gamma_s)=|t-s|\,\sfd(\gamma_0,\gamma_1)\) for every \(t,s\in[0,1]\). We denote by \({\rm Geo}(\X)\) the (closed sub-) set of all
geodesics in \(C([0,1],\X)\). Given \(q\in[1,\infty]\) and
\(\mu_0,\mu_1\in\mathscr P_q(\X)\), set
\[
{\rm OptGeo}_q(\mu_0,\mu_1)\coloneqq\Big\{
\pi\in\mathscr P\big({\rm Geo}(\X)\big)\,:\,
(\e_i)_\sharp\pi=\mu_i\;\forall i=0,1,\,
(\e_0,\e_1)_\sharp\pi\in{\rm Opt}_q(\mu_0,\mu_1)\Big\}.
\]
The elements of \({\rm OptGeo}_q(\mu_0,\mu_1)\) are called
\(q\)-optimal dynamical plans between \(\mu_0\) and \(\mu_1\).
For $q \in (1,\infty]$ and \(\pi\in{\rm OptGeo}_q(\mu_0,\mu_1)\), it is well-known that \(t\mapsto(\e_t)_\sharp\pi \in \PP_q(\X)\) \ is a $W_q$-geodesic and\begin{subequations}\begin{align}\label{eq:dynOPq}
{\rm Ke}^{1/q}_q(\pi)=W_q(\mu_0,\mu_1),&\quad\text{ if }q<\infty,\\
\label{eq:dynOPinfty}
{\rm Lip}(\pi)=W_\infty(\mu_0,\mu_1),&\quad\text{ if }q=\infty.
\end{align}\end{subequations}
See \cite[Remark 2.11]{NobiliPasqualettoSchultz22} for a proof when $q=\infty$. From \eqref{eq:dynOPinfty}, we directly get
\begin{equation}\label{eq:Winf_admissible}
W_\infty\big((\e_{t_{i}})_\sharp \eta, (\e_{t_i+1})_\sharp \eta\big)  \le (t_{i+1}-t_i)\Lip(\eta),
\end{equation}
for every $\eta \in \PP((C[0,1],\X))$. We shall also deal with the so-called $M$-polygonal geodesic plans for $M\in \N$, i.e.\ $\pi \in \PP(C([0,1],\X))$ so that, for some $q\in(1,\infty]$ and $0=t_0<t_1<\ldots <t_M=1$, we have
\[
\left( {\sf rest}_{t_i}^{t_{i+1}}\right)_\sharp \pi \in {\rm OptGeo}_q( (\e_{t_i})_\sharp \pi, (\e_{t_i+1})_\sharp \pi),\qquad\forall i=0,...,M-1,
\]
where ${\sf rest}_t^s \colon C([0,1],\X) \to C([0,1],\X)$ is defined as ${\sf rest}_t^s(\gamma)(r) \coloneqq \gamma_{t(1-r) + sr}$. For later use and for $\pi$ as above, we recall (see \cite[Remark 3.4]{NobiliPasqualettoSchultz22}) the following identity
\begin{equation}\label{eq:Lippolygonal}
\Lip(\pi) = \max_i \frac{1}{t_{i+1}-t_i}\Lip\left(\left( {\sf rest}_{t_i}^{t_{i+1}}\right)_\sharp \pi \right).
\end{equation}
\subsection{Sobolev and BV calculus}
Denote the local Lipschitz constant of $f\colon \X\to\R$ as
\[ 
\lip \, f(x) \coloneqq \limsup_{y\to x}\frac{|f(y)-f(x)|}{\sfd(y,x)}
\]
set to $0$ if $x$ is isolated. We recall the definition of Sobolev space via relaxation, following \cite{AmbrosioGigliSavare11-3} (we also refer to \cite{Cheeger00,Shanmugalingam00} and to \cite{AmbrosioIkonenLucicPasqualetto24} for a complete discussion).
\begin{definition}[Sobolev space]\label{def:Chp}
Let $\Xdm$ be a metric measure space and $p \in (1,\infty)$. Define the $p$-Cheeger energy $ {\rm Ch}_p \colon L^p(\mm)\to [0,\infty]$ 
\[ {\rm Ch}_p(f) \coloneqq  \inf \Big\{ \liminf_{n\to\infty} \int \lip^p  f_n\, \d \mm \colon (f_n) \subset L^p(\mm)\cap \Lip(\X), f_n\to f \emph{ in } L^p(\mm) \Big\}.\]
The Sobolev space is defined as $W^{1,p}(\X)\coloneqq \{f \in L^p(\mm) \colon {\rm Ch}_p(f) <\infty\}$.
\end{definition}
We recall the following lower semicontinuity property
\begin{equation}\label{eq:W1psemicontinuous}
\begin{array}{l}
(f_n)\subset W^{1,p}(\X),\, f_n \to f \text{ in }L^p \\
 \liminf_{n\uparrow\infty} {\rm Ch}_p(f_n)<\infty
\end{array} \quad \Rightarrow \quad 
\begin{array}{l}
f \in W^{1,p}(\X) \\
 {\rm Ch}_p(f)\le \liminf_{n\uparrow\infty} {\rm Ch}_p(f_n),
\end{array} 
\end{equation}
and that there exists $| D f|_p \in L^p(\mm)$ so that
\[ {\rm Ch}_p(f) = \int |D f|_p^p\, \d \mm.\]
See \cite{AmbrosioGigliSavare11-3} for these claims. The subscript $p$ is intentional, as the object $|Df|_p$ may depend on $p$ \cite{DiMarinoSpeight13}. 

We recall a basic Leibniz rule and, since it is slightly non standard, we include a proof.
\begin{lemma}
    Let $p\in(1,\infty)$, let $\Xdm$ be a metric measure space and let $f,g \in W^{1,p}(\X)$. If $g,|Dg|_p \in L^\infty(\mm)$, then $fg \in W^{1,p}(\X)$ and it holds
    \begin{equation}\label{eq:leibniz modificata}
        \rmCh^{1/p}_p(fg) \le \|g\|_{L^\infty(\mm)}\rmCh_p^{1/p}(f) + \||Dg|_p\|_{L^\infty(\mm)}\|f\|_{L^p(\mm)}.
    \end{equation}
\end{lemma}
\begin{proof}
    Since $g$ is bounded, then $fg \in L^p(\mm)$. Let $N\in \N$ and consider $f^N\coloneqq (-N)\vee f \wedge N$. In this case, $f^Ng \in W^{1,p}(\X)$ by the standard Leibniz rule (see, e.g.\ \cite{AmbrosioIkonenLucicPasqualetto24}) and
    \[
        \rmCh^{1/p}_p(f^Ng) \le \|g\|_{L^\infty(\mm)}\rmCh_p^{1/p}\left(f^N\right) + \||Dg|_p\|_{L^\infty(\mm)}\left\|f^N\right\|_{L^p(\mm)},\qquad\forall N\in N.
    \]
    The conclusion then follows by a chain rule argument, and by lower semicontinuity \eqref{eq:W1psemicontinuous}.
\end{proof}

We now pass to the introduction of the total variation and the BV-space via relaxation (\cite{Miranda03,AmbrosioDiMarino14}).
\begin{definition}
Let $\Xdm$ be a metric measure space and let $f \in L^1(\mm)$. We define
\[ |\dD f|(\X):= \inf \Big\{ \liminf_{n\to \infty} \int \lip\, f_n\, \d \mm \colon f_n \subset \LIP_{loc}(\mm), f_n \to f \emph{ in } L^1(\mm) \Big\}.\]
The BV-space is defined as $BV(\X) \coloneqq \{ f \in L^1(\mm) \colon |\dD f|(\X)<\infty\}$.
\end{definition}
The quantity $|\dD f|(\X)$ is called the total variation and satisfies the  lower-semicontinuity property
\begin{equation}\label{eq:BVsemicontinuous}
\begin{array}{l}
(f_n)_n \subset BV(\X), f_n \to f \text{ in }L^1 \\
 \liminf_{n\uparrow\infty} |\dD f_n|(\X)<\infty
\end{array} \quad \Rightarrow \quad 
\begin{array}{l}
f \in BV(\X) \\
 |\dD f|(\X)\le \liminf_{n\uparrow\infty} |\dD f_n|(\X).
\end{array} 
\end{equation}

The space $BV(\X)$ can be equivalently characterized in duality with $\infty$-test plans \cite{AmbrosioDiMarino14} (see also \cite{Martio16-2,NobiliPasqualettoSchultz22,BrenaNobiliPasqualetto22} for other equivalent approaches). {The next result is taken from \cite[Theorem 4.5.3]{DiMarinoPhD}.} We say that a test plan $\pi$ is well-contained in an open set $\Omega \subset \X$ if ${\rm dist}([\pi],\Omega^c)>0$, where $[\pi]\coloneqq \{\gamma_t \colon \gamma \in \supp(\pi),t\in[0,1]\}$. 
\begin{theorem}\label{thm:BVplans}
Let $\Xdm$ be a metric measure space and $f \in L^1(\mm)$. They are equivalent:
\begin{itemize}
\item[(i)] $f \in BV(\X)$;
\item[(ii)] There exists $C\ge 0$ so that
\[ \int f(\gamma_1)-f(\gamma_0)\,\d \pi \le {\rm Comp}(\pi)\Lip(\pi)C,\]
for all $\infty$-test plans $\pi$.
\end{itemize}
Moreover, denoting $\cC_f$ the minimal constant $C\ge 0$ for (ii) to hold, we have $\cC_f = |\dD f|(\X)$. Finally, for every $\Omega \subset \X$ open, the following identification holds:
\[
    |\dD f|(\Omega) = \sup \, \frac{\int f(\gamma_1)-f(\gamma_0)\,\d \pi}{{\rm Comp}(\pi)\Lip(\pi)},
\] 
where the supremum is taken among all $\infty$-test plan $\pi$ well contained in $\Omega$. 
\end{theorem}
\begin{remark}\label{rmk:BVplanbounded}
\rm
We notice that (ii) in Theorem \ref{thm:BVplans} is equivalent to 
\begin{itemize}
\item[(ii')] There exists $C\ge 0$ so that
\[ \int f(\gamma_1)-f(\gamma_0)\,\d \pi \le {\rm Comp}(\pi)\Lip(\pi)C,\]
for all $\infty$-test plans $\pi$ with bounded support.
\end{itemize}
Clearly, we only need to show that (ii') implies (ii). This is a standard approximation argument considering, for $\pi$  arbitrary $\infty$-test plan, e.g.\  the boundedly supported test plans
\[
\Gamma_n \coloneqq \big\{ \gamma \in \Lip([0,1],\X) \colon \sfd(\gamma_0,\bar x) \le n, \Lip(\gamma)\le n\big\} \subseteq C([0,1],\X),\qquad \pi_n \coloneqq \frac{\pi\res\Gamma_n}{\pi(\Gamma_n)}.
\]\fr 
\end{remark}
\subsection{pmG-topology and convergence of functions}
We recall here some basic facts about convergence of metric measure structures \cite{Gromov07} (see also \cite{Sturm06I}, here we follow the extrinsic approach described in \cite{GMS15}). A pointed metric measure space is a quadruple $\Xdmx$, where $\Xdm$ is a metric measure space and $x  \in \X$. Also, set $\bar{\N}:=\N\cup\{\infty\}$.
\begin{definition}[pmG-convergence]\label{def:pmGH}
Let $\Xdmxn$, $n \in \bar{\N}$, be a sequence of pointed metric measure spaces. We say that that $\Xdmxn$ \emph{pointed measured Gromov converges} (pmG-converges for short) to $\Xdmxinf$ provided there exists a complete and separable metric measure space $(\Z,\sfd)$, called \emph{realization}, and isometric embeddings
\[
\begin{split}
\iota_n &\colon (\X_n,\sfd_n) \to (\Z,\sfd), \\
\iota_\infty &\colon (\X_\infty,\sfd_\infty) \to (\Z,\sfd),
\end{split}
\]
such that $(\iota_n)(x_n) \to \iota_\infty(x_\infty)$ and
\[ (\iota_n)_\sharp \mm_n \to (\iota_\infty)_\sharp \mm_\infty, \qquad \text{in duality with }C_{bs}(\Z).\]
In this case, we shortly write $\X_n \overset{pmG}{\rightarrow} \X_\infty$.
\end{definition}
In what follows, we shall identify the spaces $\X_n$, $n\in\bar\N$, with their isomorphic images in $\Z$, i.e.\ adopting the so-called extrinsic approach. See \cite{GMS15} for details and equivalences with other notions of convergences. We recall the definition of convergence of functions on varying spaces referring to \cite{GMS15,AmbrosioHonda17}.
\begin{definition}[$L^p$-convergence]\label{def:Lp convergence pmGH}
Let $\Xdmxn$, $n \in \bar{\N}$, be a sequence of pointed metric measure spaces satisfying $\X_n \overset{pmG}{\rightarrow} \X_\infty$. Fix a realization $(\Z,\sfd)$. For $p \in (1,\infty)$, we say that:
\begin{itemize}
\item[(i)] $f_n\in L^p(\mm_n)$ \emph{converges $L^p$-weak} to $f_\infty\in L^p(\mm_\infty)$, provided $\sup_{n}\|f_n\|_{L^p(\mm_n)}<\infty$ and $f_n\mm_n \weakto f_\infty\mm_\infty$ in duality with $C_{bs}(\Z)$,
\item[(ii)] $f_n\in L^p(\mm_n)$ \emph{converges $L^p$-strong} to $f_\infty\in L^p(\mm_\infty)$, provided it converges $L^p$-weak and $\limsup_n \|f_n\|_{L^p(\mm_n)} \le  \|f_\infty\|_{L^p(\mm_\infty)}$.
\end{itemize}
For $p=1$, we say:
\begin{itemize}
\item[(i')] $f_n\in L^1(\mm_n)$ \emph{converges $L^1$-weak} to $f_\infty\in L^1(\mm_\infty)$, provided $\sup_{n}\|f_n\|_{L^1(\mm_n)}<\infty$ and $f_n\mm_n \weakto f_\infty\mm_\infty$ in duality with  $C_{bs}(\Z)$;
\item[(ii')] $f_n\in L^1(\mm_n)$ \emph{converges $L^1$-strong} to $f_\infty\in L^1(\mm_\infty)$, provided $\sigma \circ f_n$ converges $L^2$-strong to $\sigma \circ f_\infty$, where $\sigma(z) = {\rm sgn}(z)\sqrt{|z|}$.
\end{itemize}
\end{definition}
The case of $L^1$-strong convergence needs the above modification, otherwise, it does not match the standard convergence for fixed-based space \cite[Remark 1.27]{AmbrosioBrueSemola19}. Here we recall an important property of the weak/strong convergence of integral couplings: given $p,q \in (1,\infty)$ H\"older conjugate, if $f_n$ converges $L^p$-strong to $f_\infty$ and $g_n$ converges $L^q$-weak to $g_\infty$, then 
\begin{equation}
\lim_{n\to\infty}\int f_n g_n\,\d\mm_n = \int f_\infty g_\infty\,\d\mm_\infty.\label{eq:pq coupling}
\end{equation}
(see, e.g., \cite[Proposition 6.2]{NobiliViolo22} for a proof of this fact working also for non-compact spaces).

Next, we recall the existence of recovery sequences. As already discussed in the Introduction, the validity of the following result is discussed in \cite{Gigli23_working} and its proof for general exponent follows by the analysis in \cite[Theorem 8.1]{AmbrosioHonda17}.
\begin{theorem}\label{th:limsup}
    Let $\Xdmxn$ be pointed metric measure spaces satisfying $\X_n \overset{pmG}{\rightarrow} \X_\infty$ for some $\Xdmxinf$. Then, for every $p \in [1,\infty)$ and $f_\infty \in L^p(\mm_\infty)$ it holds
    \begin{itemize}
        \item[{\rm (i)}] if $p>1$, there are $f_n \in L^p(\mm_n)$ converging in $L^p$-strong to $f_\infty$ such that
        \[
            \limsup_{n\uparrow \infty}\rmCh_p(f_n) \le \rmCh_p(f_\infty);
        \]
        \item[{\rm (ii)}] if $p=1$, there are $f_n \in L^1(\mm_n)$ converging in $L^1$-strong to $f_\infty$ such that
        \[
           \limsup_{n\uparrow \infty}|\dD f_n|(\X_n)\le  |\dD f_\infty|(\X_\infty).
        \]
    \end{itemize}
\end{theorem}

\subsection{Curvature dimension conditions}
We recall here the definition of a \({\sf CD}_q\)-space, after \cite{Sturm06I,Sturm06II} and \cite{Lott-Villani09}. We start from the infinite-dimensional version.

Recall the definition of the Shannon entropy functional ${\rm Ent}_\mm \colon \PP(\X) \to [0,\infty]$ as defined by
\begin{equation} {\rm Ent}_\mm(\mu) := \int \rho\log \rho \,\d \mm, \qquad \text{if } \mu = \rho\mm, \quad \infty\text{ otherwise}. \label{eq:Entm}
\end{equation}

\begin{definition}[\({\sf CD}_q(K,\infty)\)-space]\label{def:CDq infty}
A metric measure space \((\X,\sfd,\mm)\) is said to be a \({\sf CD}_q(K,\infty)\)
space, for some \(K\in\R\) and \(q\in(1,\infty)\), if given \(\mu_0,\mu_1\in\mathscr P_q(\X)\),
there exists \(\pi\in{\rm OptGeo}_q(\mu_0,\mu_1)\) such that, denoting
\(\mu_t\coloneqq(\e_t)_\sharp\pi\) for every \(t\in[0,1]\), it holds
\begin{equation}
{\rm Ent}_\mm(\mu_t)\leq(1-t){\rm Ent}_\mm(\mu_0)+t{\rm Ent}_\mm(\mu_1)-\frac K2t(1-t)W_q^2(\mu_0,\mu_1),
\label{eq:convexity Ent}
\end{equation}
for every $ t\in[0,1]$. When \(q=2\), we only write \({\sf CD}(K,\infty)\) in place of \({\sf CD}_2(K,\infty)\).
\end{definition}
It is important in this note to recall that, on ${\sf CD}_q$-spaces, there is an abundance of $q$-test plans obtained via Wasserstein interpolation. The following theorem is due to \cite{Rajala12-2} (see \cite[Appendix B]{GigliNobili22} for details in the case $q\neq 2$ which we also consider here).
\begin{theorem}\label{thm:rajalaCDqKN}
Let $\Xdm$ be a $\CD_q(K,\infty)$-space for some $K \in \R$ and $q \in (1,\infty)$. Let $\rho_0 ,\rho_1 \in L^{\infty}(\mm)$ be boundedly supported probability densities and suppose, for some $D\in(0,\infty)$, that
\[
\Lip(\pi)\le D,\qquad \forall \pi \in \OptGeo(\rho_0\mm,\rho_1\mm).
\]
Then, there exists $ \pi \in \OptGeo_q(\rho_0\mm,\rho_1\mm)$ with $(\e_t)_\sharp \pi \ll \mm$ and, writing $(\e_t)_\sharp \pi\coloneqq  \rho_t\mm$, we have
\[
 \| \rho_t\|_{L^\infty(\mm)} \le e^{\frac{K^-}{12}D^2} \|\rho_0\|_{L^\infty(\mm)}\vee \|\rho_1\|_{L^\infty(\mm)},\qquad \forall t \in [0,1].
 \]
\end{theorem}
We point out that, in \cite[Theorem 1.3]{Rajala12-2}, the choice $D \coloneqq \supp(\rho_0)\cup\supp(\rho_1)$ is considered. However, its proof relies on \cite[Proposition 3.11]{Rajala12-2} requiring precisely our assumptions (see also \cite{RajalaDgiustadimensionefinita}).

Enforcing a non-branching assumption on the curvature dimension condition produces useful properties. Recall that a metric space \((\X,\sfd)\) is non-branching
provided 
\begin{equation}
(\gamma_0,\gamma_t)=(\sigma_0,\sigma_t),\,\text{ for some }t\in(0,1)
\quad\Longrightarrow\quad\gamma=\sigma,
\label{eq:nonbranching}
\end{equation}
holds for every \(\gamma,\sigma\in{\rm Geo}(\X)\). Following \cite{RajalaSturm12}, we call a  metric measure space $\Xdm$  $q$-essentially non-branching for some $q \in (1,\infty)$, provided for any  $\mu_0,\mu_1 \in \PP_q(\X)$ there exists $\pi \in {\rm OptGeo}_q(\mu_0,\mu_1)$ that is concentrated on a set on non-branching geodesics. 
\begin{remark}[Improved estimates for strong ${\sf CD}$]\label{rem:rajala improved} \rm
If, in Theorem \ref{thm:rajalaCDqKN}, we assume that $\Xdm$ is a strong $\CD_q(K,\infty)$ space, then it turns out that improved density estimates are available. The strong $\CD_q(K,\infty)$ condition requires \eqref{eq:convexity Ent} \emph{for all} dynamical plans. As a byproduct of this reinforcement, it was proved in \cite{RajalaSturm12} (for $q=2$, but extendable for every $q\in(1,\infty)$) that
\begin{equation}\label{eq:optgeo unique}
\OptGeo_q(\mu_0,\mu_1)\text{ is a singleton for all absolutely continuous } \mu_0,\mu_1\in\PP_q(\X).
\end{equation}
In particular, the strong ${\sf CD}$ property guarantees in Theorem \ref{thm:rajalaCDqKN} that the unique $\pi \in \OptGeo_q(\rho_0\mm,\rho_1\mm)$ satisfies
\begin{equation}
 \|\rho_t\|_{L^\infty(\mm)} \le e^{\frac{K^-}{12}\Lip( \pi)^2} \|\rho_0\|_{L^\infty(\mm)}\vee \|\rho_1\|_{L^\infty(\mm)}, \label{eq:rajala improved}
\end{equation}
for every $t \in [0,1]$. This is trivially true since we can choose $D=\Lip( \pi) = \|\sfd(\e_0,\e_1)\|_{L^\infty(\pi)}.$ Lastly, we shall also use property \eqref{eq:rajala improved} in the case of $q$-essentially nonbranching ${\sf CD}_q(K,N)$ spaces when $N<\infty$, a class that will be shortly introduced. This will be possible thanks to \cite{Kell17}. \fr  
\end{remark}

We now recall the finite-dimensional curvature dimension conditions, starting from the definition of the \(\sigma_{K,N}^{(t)}\)  coefficient for \(K\in\R\),
\(N\in(1,\infty)\), \(t\in[0,1]\), and \(\theta\in[0,+\infty)\), defined as  
\[
\sigma^{(t)}_{K,N}(\theta) := \begin{cases} 
\ \infty, &\text{if } K\theta^2 \ge N\pi^2,\\
\ \frac{\sin(t\theta\sqrt{K/N})}{\sin(\theta\sqrt{K/N)}}, &\text{if } 0<K\theta^2<N\pi^2,\\
\ t, &\text{if } K\theta^2<0\text{ and } N=1 \text{ or if } K\theta^2=0, \\
\ \frac{\sinh(t\theta\sqrt{-K/N})}{\sinh(\theta\sqrt{-K/N)}}, &\text{if } K\theta^2< 0 \text{ and }N>0.\\
\end{cases}
\]
We then set $\tau^{(t)}_{K,N}(\theta) := t^{\frac{1}{N}}\sigma^{(t)}_{K,N-1}(\theta)^{1-\frac{1}{N}}$ while $\tau^{(t)}_{K,1}(\theta) =t$ if $K\le 0$ and $\tau^{(t)}_{K,1}(\theta)=\infty$ if $K>0$. 

Given a metric measure space \((\X,\sfd,\mm)\) and
\(N\in(1,\infty)\), we define the \(N\)-R\'{e}nyi relative entropy functional
\(\mathcal U_N\colon\mathscr P(\X)\to[0,+\infty]\) as
\[
\mathcal U_N(\mu)\coloneqq\int\rho^{1-\frac{1}{N}}\,\d\mm,
\qquad \forall \mu\in\mathscr P(\X),\;\mu=\rho\mm+\mu^s
\text{ with }\mu^s\perp\mm.
\]
\begin{definition}[\({\sf CD}_q(K,N)\)-space]\label{def:CDq N}
A metric measure space \((\X,\sfd,\mm)\) is said to be a \({\sf CD}_q(K,N)\)
space, for some \(K\in\R\), \(N\in(1,\infty)\) and \(q\in(1,\infty)\), if given any  \(\mu_0=\rho_0\mm,\mu_1=\rho_1\mm \in\mathscr P_q(\X)\), there exists \(\pi\in{\rm OptGeo}_q(\mu_0,\mu_1)\) so that, denoting
\(\mu_t\coloneqq(\e_t)_\sharp\pi\ll\mm\), it holds
\[
\mathcal U_{N'}(\mu_t)\geq\int\rho_0(\gamma_0)^{-\frac{1}{N'}}
\tau_{K,N'}^{(1-t)}\big(\sfd(\gamma_0,\gamma_1)\big)+
\rho_1(\gamma_1)^{-\frac{1}{N'}}
\tau_{K,N'}^{(t)}\big(\sfd(\gamma_0,\gamma_1)\big)\,\d\pi(\gamma),
\]
for every \(N'\geq N\) and \(t\in[0,1]\). When \(q=2\),
we only write \({\sf CD}(K,N)\) in place of \({\sf CD}_2(K,N)\).
\end{definition}
We shall use the well-known fact that the finite-dimensional $\CD_q$ condition implies its infinite-dimensional analogue. We recall next from \cite[Theorem 1.1]{ACCMcS20} the following important result.
\begin{theorem}[Equivalence of \({\sf CD}_q\) on \(q>1\)]
\label{thm:CDq independent}
Let \((\X,\sfd,\mm)\) be $q$-essentially non-branching for all $q\in(1,\infty)$. Assume that $\X$ is a \({\sf CD}(K,N)\) space,
for some \(K\in\R\) and \(N\in(1,\infty)\) and the measure
\(\mm\) is finite. Then \((\X,\sfd,\mm)\) is a \({\sf CD}_q(K,N)\)
space for all \(q\in(1,\infty)\).
\end{theorem}

Finally, we report the definition of the measure contraction property defined in \cite{Ohta07} and \cite{Sturm06II}.
\begin{definition}[$\mcp$-spaces]\label{def:MCP}
A metric measure space $\Xdm$ is said to be ${\sf MCP}(K,N)$ for some $K \in \R, N \in [1,\infty)$ if for any $\mu_0=\rho_0\mm \in \PP_2(\X)$ absolutely continuous with bounded support contained in $\supp(\mm)$ and $o \in \supp(\mm)$, there exists $\pi \in \OptGeo_2(\mu_0,\delta_o)$ so that
\begin{equation}
    \mathcal{U}_N((\e_t)_\sharp \pi) \ge \int \tau^{(1-t)}_{K,N}(\sfd(x,o))\rho_0^{1-\frac 1N}\, \d \mm,\qquad \forall t \in [0,1).\label{eq:mcp convexity}
\end{equation}
\end{definition}
When coupling the ${\sf MCP}$-class with the essentially non-branching condition, the ${\sf MCP}$ notions of \cite{Ohta07} and \cite{Sturm06II} coincide.  Also, we point out that the $\mcp$ condition is independent of the transport exponent, i.e.\ \eqref{eq:mcp convexity} holds also for $W_q$-geodesics with $q\neq 2$ (see, e.g., \cite[Remark 5.2]{GigliNobili22}).

Finally, we recall the validity of the Sobolev-Poincar\'e inequality under curvature dimension conditions. If $\Xdm$ is either a ${\sf CD}(K,N)$ space, or a $q$-essentially nonbranching ${\sf MCP}(K,N)$ space for some $q \in (1,\infty),\, K\in\R,\, N\in [1,\infty)$, then for all $R>0$ there is a constant $C(K,N,R)>0$ so that, for all $r\le R, \, x \in \X$, it holds
\[
    \left(\fint_{B_r(x)}|u-u_{B_r(x)}|^{s}\,\d \mm \right)^{\frac {1}{s}}\le C(K,N,R)r\fint_{B_{2r}(x)}\lip(u)\,\d \mm,\qquad \forall u \in \Lip(\X),
\]
where $u_B = \fint_B u\,\d \mm$ for $B\subset \X$ and $1\le s <1^* = N/(N-1)$. The above simply follows by \cite[Theorem 5.1]{HK00} as the spaces under considerations satisfy the Bishop-Gromov inequality  (\cite{Sturm06II}) and are equipped with a weak local Poincar\'e inequality (see \cite{Rajala12-2} and \cite{VonRenesse08}).  By considering optimal approximations in the definition of the total variation, the above yields
\[
    \left(\fint_{B_r(x)}|u-u_{B_r(x)}|^{s}\,\d \mm \right)^{\frac {1}{s}}\le C(K,N,R)r\frac{|\dD u|(B_{2r}(x))}{\mm(B_{2r}(x))},\qquad \forall u \in BV(\X).
\]
After algebraic manipulations, we get
\begin{equation}\label{SobPoincare BV}
    \|u\|_{L^s(B_r(x))}\le C(K,N,R)r\frac{\mm(B_r(x))^{1/s}}{\mm(B_{2r}(x))} |\dD u|(B_{2r}(x)) + \mm(B_r(x))^{1/s-1}\|u\|_{L^1(B_r(x))}.
\end{equation}
\section{An equivalent Lagrangian notion of calculus}
Here we follow the strategy of \cite{AmbrosioGigliSavare11-3} to prove that the relaxed notion of Sobolev space of Definition \ref{def:Chp} is equivalent to a Lagrangian notion controlling the oscillations of the function via test plan superposition. Differently from \cite{AmbrosioGigliSavare11-3}, we prove the equivalence with an ``integrated'' notion which is strongly inspired by \cite{AmbrosioDiMarino14}.
\begin{theorem}\label{thm:Sobolevintegrated}
Let $\Xdm$ be a metric measure space, let $p,q \in (1,\infty)$ be H\"older conjugate and $f \in L^p(\mm)$. The following are equivalent:
\begin{itemize}
\item[(i)] $f \in W^{1,p}(\X)$;
\item[(ii)] There exists $C\ge 0$ so that
\[ \int f(\gamma_1) - f(\gamma_0)\,\d \pi \le {\rm Comp}(\pi)^{1/p} \rmKe_q^{1/q}(\pi) C,\]
for all $q$-test plans $\pi$.
\end{itemize}
Moreover, denoting $C_f$ the minimal constant $C\ge0$ for (ii) to hold, we have 
$C_f^p = {\rm Ch}_p(f)$. 
\end{theorem}
\begin{proof}
The implication \textrm{(i) $\to$ (ii)} is immediate noticing that for any sequence $(f_n) \subseteq \Lip(\X)$ competitor in the definition of ${\rm Ch}_p(f)$, we have for any $n \in \N$ that
\[
 \int f_n(\gamma_1)-f_n(\gamma_0)\,\d \pi \le \iint_0^1 \lip\, f_n(\gamma_t)|\dot \gamma_t|\,\d t \d \pi \le {\rm Comp}(\pi)^{1/p} \rmKe_q^{1/q}(\pi)\| \lip f_n\|_{L^p(\mm)}.
\]
Passing to the limit in the above (recall $L^p(\mm) \ni f \mapsto f\circ \e_t \in L^1(\pi)$ is continuous for any $t \in [0,1] $, see e.g.\ \cite{GP19}) and optimizing over $(f_n)$, we get (ii) holds true and $C_f^p \le {\rm Ch}_p(f)$. We pass to the converse implication \textrm{(ii) $\to$ (i)} and ${\rm Ch}_p(f) \le C_f^p$. We subdivide the proof into different steps.

\noindent{\bf Reduction to $f\in L^\infty(\mm)$}. We observe that it is not restrictive to suppose $f \in L^\infty(\mm)$. To see that, suppose that the implication is valid for bounded functions and let $f \in L^p(\mm)$ satisfy (ii) with $C_f<\infty$. 
 For all $N\in\N$ set $f^N \coloneqq (-N)\vee f \wedge N$ and, for any $q$-test plan $\pi$, consider the Borel set
\[
    \Gamma^+_N \coloneqq\{ \gamma: f^N(\gamma_1)-f^N(\gamma_0) >0\}.
\]
For any curve $\gamma \in \Gamma^+_N$, it holds that
\[
    f^N(\gamma_1)-f^N(\gamma_0)=|f^N(\gamma_1)-f^N(\gamma_0)| \leq |f(\gamma_1)-f(\gamma_0)|=f(\gamma_1)-f(\gamma_0),
\]
since $t\mapsto (-N)\vee t \wedge N$ is nondecreasing and $1$-Lipschitz. If $\pi(\Gamma^+_N)>0$, we can set $\pi^+_N \coloneqq \pi\res{\Gamma^+_N}/\pi(\Gamma^+_N)$ and estimate
\begin{equation}\label{eq:truncation plans}
\begin{aligned}
    \int f^N(\gamma_1)-f^N(\gamma_0)\,\d \pi &\le \int f^N(\gamma_1)-f^N(\gamma_0)\,\d \pi\res{\Gamma^+_N} = \pi(\Gamma^+_N)\int f^N(\gamma_1)-f^N(\gamma_0)\,\d \pi^+_N \\
    &\le \pi(\Gamma^+_N) \int f(\gamma_1)-f(\gamma_0)\,\d \pi^+_N \le \pi(\Gamma^+_N) {\rm Comp}^{1/p}(\pi^+_N)\rmKe_q^{1/q}(\pi^+_N)C_f \\
    &\le {\rm Comp}(\pi)^{1/p}\rmKe_q^{1/q}(\pi)C_f,
\end{aligned}
\end{equation}
having used the assumption (ii), i.e.\  $C_f<\infty$, the simple observations
\[
    {\rm Comp}(\pi^+_N) \le  \pi(\Gamma^+_N)^{-1}{\rm Comp}(\pi),\qquad \rmKe_q(\pi^+_N) \le   \pi(\Gamma^+_N)^{-1}\rmKe_q(\pi),
\]
and the fact that $p,q$ are H\"older conjugate exponents. If, instead, $\pi(\Gamma^+_N)=0$, then \eqref{eq:truncation plans} becomes trivially valid. In particular, by arbitrariness of $\pi$, we thus get
\[
C_{f^N}\le C_f,\qquad \forall N\in\N.
\]
Therefore, assuming that (ii) is equivalent to (i) for bounded functions, we deduce
\[
    \rmCh_p(f^N)=C^p_{f^N} \le C^p_f,\qquad \forall N\in\N.
\]
As $f^N \to f $ in $L^p(\mm)$, by lower semicontinuity  \eqref{eq:W1psemicontinuous} we obtain $f \in W^{1,p}(\X)$ and $\rmCh_p(f) \le C^p_f $. This discussion guarantees that we can assume in the sequel, without loss of generality, that $f\in L^\infty(\mm)$.

\noindent\textbf{Proof when }$\mm(\X)<\infty$. Let us assume for the moment that $\mm(\X)<\infty$. This will allow us to sum constants to $f$ without losing its integrability. Being $f \in L^\infty(\mm)$, we can consider $c,C,H>0$ s.t.
\[ 
0< c +H \le f+H \le C+ H,\qquad \mm\text{-a.e.}.\]
Notice that $H>0$ can be chosen arbitrarily large (eventually we will send $H \uparrow \infty$). Now, let $m \coloneqq \int f+H \, \d \mm$, and set
\[
g_0:=m^{-1}(f+H),
\]
and consider the heat flow trajectory 
$$t \mapsto g_t:=h_t(g_0) \in L^p(\mm),$$
of the functional ${\rm Ch_p} \colon L^p(\mm) \to [0,\infty]$ (\cite{AmbrosioGigliSavare11}, see also \cite{GP19}). Having assumed $\mm(\X)<\infty$ and $f\in L^\infty(\mm)$,  then $g_0 \in L^p(\mm)$ and such heat flow trajectory can be considered. We recall from \cite[Proposition 6.6]{AmbrosioGigliSavare11-3} the properties:
\begin{itemize}
\item[ (a)](Mass preservation) $1= \int g_0 \, \d \mm = \int g_t\,\d \mm$ for any $t>0$.
\item[ (b)] (Weak maximum principle) $f\le C$ (resp. $f \ge c$) $\mm$-a.e.\ then $g_t \le m^{-1}(C+H)$ (resp. $g_t\ge m^{-1}(c+ H)$)  $\mm$-a.e..
\item[ (c)] (Energy dissipation) If $c \le f \le C$ $\mm$-a.e.\ and $\Phi \in C^2([m^{-1}(c+H),m^{-1}(C+H)])$, then $t\mapsto \int \Phi(g_t)\,\d\mm$ is locally absolutely continuous on $(0,\infty)$ and
\[
\frac{\d}{\d t}\int \Phi(g_t)\,\d\mm = -p \int \Phi''(g_t)|Dg_t|_p^p\,\d \mm,\qquad \text{a.e.\ } t.
\]
\end{itemize}

Moreover, defining 
\[ \mu_t \coloneqq  g_t \mm, \qquad \forall t \in [0,1],\]
we have that $(\mu_t)$ is an $AC^q$-curve with values in the Wasserstein space $(\PP_q(\X),W_q)$ having the following speed-control (see Kuwada's lemma \cite[Lemma 7.2]{AmbrosioGigliSavare11-3}):
\begin{equation}\label{eq:kuwada speed}
|\dot \mu_s|^q \le p^q\int \frac{|Dg_s|^p_p}{g_s^{q-1}}\,\d \mm,\qquad \text{a.e.\ }s.
\end{equation}
Recall also that, by superposition principle \cite{Lisini07}, we can consider $\pi \in \PP(C([0,1],\X))$ satisfying
\[ (\e_t)_\sharp \pi = \mu_t,\quad \forall t \in[0,1],\qquad \text{and}\qquad  \rmKe_q(\pi) = \int_0^1|\dot \mu_t|^q\,\d t\]
and we set $\pi_t:=\big({\sf rest}_0^t\big)_\sharp \pi$. By choosing $\Phi(z)=z^2/2$, we can use at first the energy dissipation principle (c) and start estimating
\[
\begin{split}
p\int_0^t \int |Dg_s|^p_p \,\d \mm \,\d s &= \int \Phi(g_0) - \Phi(g_t)\,\d \mm  \\
 &\le \int \Phi'(g_0)( g_0 - g_t)\,\d\mm \\
&=\int g_0\circ \e_0 -g_0\circ \e_1\, \d \pi_t \\
&= m^{-1} \int f(\gamma_0)-f(\gamma_1)\,\d \pi_t \\
&\le m^{-1}{\rm Comp}(\pi_t)^{\frac 1p} \rmKe_q^{\frac 1q}(\pi_t) C_f,
\end{split}
\]
having used the convexity of $\Phi$ in the first inequality and, in the latter inequality, hypothesis (ii) with the ``reversed in time'' $\big({\sf rest}_1^0\big)_\sharp \pi_t$ test plan. We now start estimating each term as follows: first, by a simple change of variable and recalling the speed-control \eqref{eq:kuwada speed}, we get
\[ 
\begin{split}
 \rmKe_q(\pi_t) &= t^{q-1}\iint_0^t|\dot \gamma_s|^q\,\d s \, \d \pi =  t^{q-1}\int_0^t|\dot \mu_s|^q\,\d s \le t^{q-1} p^q \int_0^t\int \frac{|Dg_s|^p_p}{g_s^{q-1}}\,\d \mm \,\d s \\
 &\le t^q p^q\left(\frac{c+H}{m}\right)^{1-q}\fint_0^t {\rm Ch}_p(g_s) \, \d s,
 \end{split}
\]
where for the last inequality we used the weak maximum principle (b). Then, from (b) we also infer
\[ 
{\rm Comp}(\pi_t) \le \frac{C+H}{m}.
\]
Combining all the estimates we get
\[ 
\fint_0^t {\rm Ch}_p(g_s) \, \d s \leq m^{-1} \left(\frac{C+H}{c+H}\right)^{\frac 1p} \left(\fint_0^t {\rm Ch}_p(g_s) \,\d s\right)^\frac 1q C_f.
\]
Thus, recalling also that $s\mapsto {\rm Ch}_p(g_s)$ is non-increasing, we obtain
\[
{\rm Ch}_p(g_t)^\frac 1p \leq \left(\fint_0^t {\rm Ch}_p(g_s) \,\d s\right)^\frac 1p \leq m^{-1} \left(\frac{C+H}{c+H}\right)^{\frac 1p} C_f.
\]
Now, sending $t \downarrow 0$, the lower semicontinuity \eqref{eq:W1psemicontinuous} gives
\[
m^{-p} \, {\rm Ch}_p(f)={\rm Ch}_p(g_0) \le m^{-p} \frac {C+H}{c+H} C_f^p.
\]
Finally taking also the limit $H\uparrow \infty$ yields $
{\rm Ch}_p(f) \le C_f^p.$

\noindent\textbf{Proof for general $\mm$}. Here we consider a locally bounded measure $\mm$ and conclude the proof. We start noticing that condition (ii) satisfies the following global-to-local property: if $K\subset \X$ is bounded and closed, then any $q$-test plan on the metric measure space $(K,\sfd,\mm\res{K})$ is a $q$-test plan on $\Xdm$. In particular, passing to $\mm\res{K}$-a.e.\ equality and regarding $f \in L^p(\mm\res{K})$, it holds that
\[ \int f(\gamma_1) - f(\gamma_0)\,\d \pi \le {\rm Comp}(\pi)^{1/p} \rmKe_q^{1/q}(\pi) C_f,\]
for all $q$-test plans $\pi$ on $(K,\sfd,\mm\res{K})$. For what has already been proved, this implies 
\begin{equation}
f \in W^{1,p}(K) := W^{1,p}(K,\sfd,\mm\res{K})\qquad \text{and}\qquad \rmCh_p^K(f) \le  C_f^p.
\label{eq:compact membership}
\end{equation}
Notice the key fact that here $K\subset \X$ is arbitrary and thus, the latter estimate is uniform. We can thus consider an exhaustion $K_n\uparrow \X$ of closed sets with $\mm(K_n)<\infty$ for each $n\in\N$. Here we consider for instance closed balls $K_n\coloneqq \overline{B_{r_n}}(x)$ centred at a given point $x \in \X$ for a suitable sequence $r_n\uparrow \infty$ so that $r_n-r_{n-1} > n$. For each $n\in\N$, let $(f^n_k)_k \subset \Lip(K_n)$ be optimal for the definition of $\rmCh_p^{K_n}(f)$, i.e.\ so that 
\[
f_k^n \to f \text{ in }L^p(K_n)\qquad  \text{and}\qquad  \int \lip^p(f_k^n)\,\d\mm\res{K_n} \to \rmCh^{K_n}_p(f),
\]
as $k\uparrow \infty$. Let now $\nchi^n \colon K_n \to [0,1]$ be $(1/n)$-Lipschitz cut-offs satisfying:
\[
\nchi^n \equiv 1 \text{ on }K_{n-1}, \qquad \supp(\nchi^n) \subset B_{r_{n-1}+n}(x) \subsetneq K_n.
\] 
Notice that $\nchi^n f \in W^{1,p}(K_n)$ and, denoting $|D_{(n)} g|_p$ the function representing the $p$-Cheeger energy of $g \in W^{1,p}(K_n)$, we clearly have by definition of $p$-Cheeger energies on $\X$ that $ \nchi^n f  \in W^{1,p}(\X)$ and
\[
\int |D_{(n)} (\nchi^n f)|^p_p\,\d\mm\res{K_n} = \rmCh_p^{K_n}(\nchi^n f) = \rmCh_p(\nchi^n f).
\]
(the latter equality is due to the fact that the admissible relaxations for $\nchi^n f$ are the same). Combining everything, for every $n\in\N$, we have

\begin{align*}
\rmCh_p^{1/p}(\nchi^nf) &= \left(\int |D_{(n)}(f \nchi^{n})|^p_p\,\d \mm\res{K_n}\right)^{1/p} \\
&\overset{\eqref{eq:leibniz modificata}}{\le}\left(\int_{B_{r_{n-1}+n}} |D_{(n)}f|^p_p\,\d\mm\right)^{1/p}  +\left(\frac1{n^p}\int |f|^p \d\mm\right)^{1/p}\\
&\le \liminf_{k\to\infty} \left(\int_{B_{r_{n-1}+n}} |D_{(n)}f^{n}_k|^p_p\,\d\mm\right)^{1/p}+\frac1n \|f\|_{L^p}\\
&\le \liminf_{k\to\infty} \left(\int \lip^p(f^{n}_k)\,\d\mm\res{K_{n}}\right)^{1/p}+\frac1n \|f\|_{L^p} \\
&= \left(\rmCh_p^{K_{n}}(f)\right)^{1/p}+\frac1n \|f\|_{L^p}\overset{\eqref{eq:compact membership}}{\le} \frac1n \|f\|_{L^p}+C_f.
\end{align*}
Being $n\in\N$ arbitrary and $\chi_n f \to f$ in $L^p(\mm)$, we conclude the proof by lower semicontinuity \eqref{eq:W1psemicontinuous} of $W^{1,p}(\X)$ sending $n\uparrow \infty$.
\end{proof}
\begin{remark}\label{rmk:planboundedW1p}
\rm 
We notice for future use that (ii) in Theorem \ref{thm:Sobolevintegrated} holds if and only if
\begin{itemize}
\item[(ii')] there exists $C\ge 0$ so that
\[ \int f(\gamma_1) - f(\gamma_0)\,\d \pi \le {\rm Comp}(\pi)^{1/p} \rmKe_q^{1/q}(\pi) C,\]
for all $q$-test plans $\pi$ with bounded support.
\end{itemize}
Clearly, we only need to show that (ii') implies (ii). This follows by a standard approximation argument considering, e.g. the boundedly supported $q$-test plan 
\[
\Gamma_n \coloneqq \Big\{ \gamma \in AC([0,1],\X) \colon \sfd(\gamma_0 , \bar x) <n, \int_0^1|\dot \gamma_t|^q\,\d t < n\Big\} \subset C([0,1],\X),\qquad \pi_n\coloneqq\frac{\pi\res{\Gamma_n}}{\pi(\Gamma_n)}.
\]\fr 
\end{remark}
 For every $ \Omega\subset \X$ open and $f\in L^p(\mm)$, we can define the (possibly infinite) nonnegative constant
\[
   C_f(\Omega)\coloneqq \sup \,  \frac{\int f(\gamma_1) - f( \gamma_0)\,\d \pi}{{\rm Comp}(\pi)^{1/p} \rmKe_q^{1/q}(\pi)},\qquad C_f(\varnothing)=0,
\]
where the supremum is taken among all $q$-test plans $\pi$ well-contained in $\Omega$. 
\begin{corollary}\label{cor:identification}
Let $\Xdm$ be a metric measure space, let $p,q \in (1,\infty)$ be H\"older conjugate and $f \in L^p(\mm)$. Then, for every $f\in W^{1,p}(\X)$, the Carathéodory extension of $\Omega \mapsto C^p_f(\Omega)$ to the Borel $\sigma$-algebra 
\[
   B\mapsto \lambda_f(B)\coloneqq \inf\{ C^p_f(\Omega) \colon  B\subset \Omega, \,  \Omega \text{ open} \},
\]
 is a nonnegative finite measure satisfying
\[
    \lambda_f = |Df|_p^p\mm.
\]
\end{corollary}
\begin{proof}
Let us observe that we are not claiming, directly, that $\Omega \mapsto C_f^p(\Omega)$ extends to a finite Borel measure, but we shall achieve this fact indirectly by the identification result given by Theorem \ref{thm:Sobolevintegrated}.

Let us consider an optimal sequence $(f_n)_n\subset \Lip(\X)$ with $f_n\to f,\, \lip \, f_n \to |Df|_p$ in $L^p(\mm)$ (\cite{AmbrosioGigliSavare11-3}) so that, for every plan $\pi$ well contained in $\Omega$ and for any $n \in \N$, it holds that
\[
 \int f_n(\gamma_1)-f_n(\gamma_0)\,\d \pi \le \iint_0^1 \lip\, f_n(\gamma_t)|\dot \gamma_t|\,\d t \d \pi \le {\rm Comp}(\pi)^{1/p} \rmKe_q^{1/q}(\pi)\left(\int_\Omega ( \lip f_n)^p\,\d \mm\right)^{1/p}.
\]
By taking $n$ to infinity, and by arbitrariness of $\pi$, we directly deduce $C_f^p(\Omega) \le \int_\Omega |Df|_p^p\,\d \mm$. Here is the crucial point. Thanks to this discussion, the Carathéodory extension $\lambda_f$ satisfies (pointwise on the $\sigma$-algebra) the inequality
\[
    \lambda_f(B) \le \int_B|Df|_p^p\,\d \mm,
\]
Therefore, to conclude, it is enough to show that the converse inequality holds in the class of open sets, as this gives at the same time that $\lambda_f$ is a nonnegative measure and satisfies $\lambda_f = |Df|_p^p\mm$.

Let us then fix any open set $\varnothing \neq \Omega\subset \X$, and let $A_\delta\subset \Omega$ be open so that ${\rm dist}(A_\delta,\Omega^c)>\delta$ for $\delta>0$ sufficiently small. We have
\[
    \lambda_f(\Omega) \ge C^p_{f,\delta}(\overline{A}_\delta) =\int |D_{(\delta)}f|_p^p\,\d\mm \res{\overline{A}_\delta} \ge \int |D_{(\delta)}f|_p^p\,\d\mm \res{ A_{2\delta}} \ge \int_{A_{2\delta}}|Df|_p^p\,\d\mm ,
\]
where $C^p_{f,\delta}(\overline{A}_\delta),|D_{(\delta)} f|_p^p$ are obtained with respect to the metric measure space $(\overline A_\delta,\sfd,\mm\res{\overline{A}_\delta})$ and $f\in L^p(\mm\res{\overline A_\delta})$. Indeed, the first inequality holds as the constant $C_f(\Omega)$ is a competitor for $C_{f,\delta}(\overline A_\delta)$ (by the fact that a test plan in $(\overline A_\delta,\sfd,\mm\res{\overline A_\delta})$ is a test plan in $\X$ that is well-contained in $\Omega$), the second equality holds by Theorem \ref{thm:Sobolevintegrated}, and the last inequality holds since
\[
\int |D_{(\delta)}f|_p^p\,\d\mm \res{ A_{2\delta}} = \lim_{n\to \infty}\int_{A_{2\delta}} \lip_{(\delta)}( f_n ) ^p\,\d \mm  \ge \liminf_{n\to \infty}\int_{A_{2\delta}} |Df_n|_p^p \,\d \mm \ge \int_{A_{2\delta}} |Df|^p_p\,\d\mm,
\]
where $f_n \to f,\lip_{(\delta)} f_n \to |D_{(\delta)}f|_p$ in $L^p(\mm\res{\overline{A}_\delta})$ is an optimal sequence $(f_n)\subset \Lip(\overline A_\delta)$ for $f \in W^{1,p}(\overline A_\delta)$, and $\lip_{(\delta)} f_n$ is the local Lipschitz constant in  $(\overline A_\delta,\sfd)$ that, pointwise on $A_{2\delta}$ satisfies $\lip_{(\delta)} f_n = \lip f_n$. In the second inequality, we used the fact that $|Df_n|_p\le \lip f_n$ $\mm$-a.e.\ by minimality, and lastly the lower semicontinuity on open subsets. Since $A_\delta \uparrow\Omega$ as $\delta\downarrow 0$, we finally deduce $\lambda_f(\Omega)\ge \int_\Omega |Df|_p^p\,\d\mm$, concluding the proof.
\end{proof}
By a common abuse of notation, we shall  from now on suppress the notation $\lambda_f$ and only write $C_f(\cdot)$ and $C_f = C_f(\X)$.
\section{Polygonal interpolations on varying spaces}
In this section, we study approximation techniques in the Wasserstein space $(\PP_q(\X),W_q)$ when the base space varies and satisfies a uniform curvature dimension condition. We shall build up suitable polygonal approximations of test plans satisfying precise quantitative compression estimates and Kinetic energy conservation principles. We distinguish between the case $q \in(1,\infty)$ and $q=\infty$, as the two analyses present some key differences.

\subsection{The case $q<\infty$}
We develop a standard, yet important for our goals, construction to approximate bounded and boundedly supported probability measures of the limit space.
\begin{lemma}\label{lem:approx muinf} 
 Let $\Xdmxn$ be a sequence of pointed metric measure spaces pmG-converging to some $\Xdmxinf$. Fix also a realization $(\Z,\sfd)$. Let $\rho\in L^\infty(\mm_\infty)$ be a probability density with ${\rm supp}(\rho)\subset B$ for some bounded set $B\subset \Z$.  Fix $\eps>0$, and let $I_\eps(B):=\{z \in \Z: \d(z,B) <\eps \}$. Then, for all $n \in \N$ there is a probability density $\rho_n \in L^\infty(\mm_n)$ so that:
\begin{subequations}\begin{align}
&\supp(\rho_n)\subset  I_\eps(B), \label{eq:supports B}\\
&\limsup_{n\uparrow\infty}\|\rho_n\|_{L^\infty(\mm_n)}\leq \|\rho\|_{L^\infty(\mm_\infty)},\label{eq:limsdens}\\
&\rho_n \to\rho \qquad L^q\text{-strong for all } q\in[1,\infty). \label{eq:Lp strong rho}
\end{align}\end{subequations}
\end{lemma}
\begin{proof}
Let $\nchi \colon \Z \to [0,1]$ be a Lipschitz cut-off with $\nchi \equiv 1$ on $B$ and $\supp(\nchi)\subset I_\eps(B)$. For instance, consider
\[
\nchi \coloneqq (1- 2\sfd(\cdot,B) /\eps) \vee 0.
\]
Set $L \coloneqq \|\rho\|_{L^\infty(\mm_\infty)}$. By \cite{GMS15}, there is $g_n \in L^2(\mm_n)$ that is $L^2$-strong converging to $\rho$. Define
\[
\tilde g_n \coloneqq  0 \vee (\nchi g_n) \wedge L,\qquad \rho_n \coloneqq \frac{\tilde g_n }{\|\tilde g_n\|_{L^1(\mm_n)}},
\]
for all $n\in\N$ when the latter makes sense. We claim that $\rho_n$ is well-defined for all $n$ large enough and does the job. By \cite[a),c) in Proposition 3.3]{AmbrosioHonda17}, we have that $\tilde g_n$ converges $L^2$-strong to $\rho$. We claim now that $ \varphi \circ \tilde g_n$ converges $L^2$-strong to $\varphi \circ \rho$ for every $\varphi \in C(\R)$ with $\varphi(0) =0 $. This will be possible since supports are uniformly bounded. From the characterization of $L^2$-strong convergence via weak convergence of graphs \cite[Prop. 5.4.4]{AmbrosioGigliSavare08} (see also \cite[Eq. (6.6)]{GMS15}), we can equivalently prove 
\begin{equation}
\int \xi(x,\varphi(\tilde g_n))\,\d\mm_n\to \int \xi(x,\varphi(\rho))\,\d \mm_\infty
\label{eq:claim 2 strong}
\end{equation}
as $n\uparrow\infty$ for every $\xi \in C(\X\times \R)$ with $|\xi(x,t)|\le \phi(x)+C|t|^2$ for some $\phi \in C_{bs}(\X)$ positive and $C>0$. Notice that we can assume $\xi(x,0)=0$ with no loss of generality (if (\ref{eq:claim 2 strong}) holds for $\xi(x,t)-\xi(x,0)$, then it holds for $\xi$). 
Thus, consider any such $\xi \in C(\X\times\R)$, and notice that, under our assumptions, choosing $\overline{\nchi}\in {\rm Lip}_{bs}(Z)$ with $\overline{\nchi}\equiv 1$ on $2B$, the function $ \bar \xi(x,t) \coloneqq \overline{\nchi}(x)\xi ( x, \varphi (-L \vee t \wedge L)) \in C(\X\times\R)$ satisfies (for possibly a bigger $C>0$)
\[
 |\bar \xi(x,t)| \le \overline{\nchi}(x)(\phi(x) + C|\varphi(-L \vee t \wedge L)|^2) \le \overline{\nchi}(x)(\phi(x)+C) + C|t|^2, 
\]
as $|\varphi(t)|^2(1+|t|^2)^{-1}$ is bounded in $[-L,L]$.
Thus, since $\tilde g_n$ converges $L^2$-strong to $\rho$ again by the characterization via graphs, we have
\[
\int \bar \xi (x,\tilde g_n) \,\d\mm_n \to  \int \bar \xi(x,\rho)\,\d \mm_\infty,
\]
as $n\uparrow \infty$. Rewriting the integral, this is precisely \eqref{eq:claim 2 strong} and the claim is therefore proved.

Now, taking $\varphi(z) = |z|^{q/2}$ if $q \in (1,\infty)$ and $\varphi(z) = {\rm sgn(z)}\sqrt{|z|}$ if  $q=1$, we see that $L^2$-strong convergence of $\varphi \circ \tilde g_n$ to $\varphi \circ \rho$ follows using the previous claim, for all $q \in [1,\infty)$. This immediately implies that $\tilde g_n$ converges $L^q$-strong to $\rho$ for every $q \in [1,\infty)$, by definition. The proof is concluded since, by $L^1$-strong convergence, we deduce that, for all $n$ large enough, $\rho_n$ is a well-defined probability density satisfying \eqref{eq:supports B} by construction. The conclusions \eqref{eq:limsdens},\eqref{eq:Lp strong rho} follow from those of $\tilde g_n$.
\end{proof}
We next build polygonal geodesic interpolation along varying spaces with key estimates. We start with an easy case when the spaces under consideration have nonnegative curvature.
\begin{proposition}\label{prop:q_polygonal K>0}
    Let $q\in(1,\infty)$ be fixed and let $\Xdmxn$ be a sequence of pointed metric measure spaces satisfying ${\sf CD}_q(0,\infty)$. Assume that $\X_n \overset{pmG}{\rightarrow} \X_\infty$ for some $\Xdmxinf$. 

    Then, for every $q$-test plan $\eta  \in \PP(C([0,1],\X_\infty))$ with bounded support, there are $q$-test plans $\pi_n \in \PP(C([0,1],\X_n))$ satisfying
    \[
    \limsup_{n\uparrow\infty} \Comp(\pi_n)\le \Comp(\eta),\qquad \limsup_{n\uparrow\infty}\rmKe_q(\pi_n)\le \rmKe_q(\eta),
    \]
    and with the property
    \[
    \frac{\d (\e_0)_\sharp \pi_n}{\d \mm_n}\to \frac{\d (\e_0)_\sharp \eta}{\d \mm_\infty},\qquad  \frac{\d (\e_1)_\sharp \pi_n}{\d \mm_n}\to \frac{\d (\e_1)_\sharp \eta}{\d \mm_\infty},\qquad   \text{in }L^q\text{-strong}.
    \]
\end{proposition}
\begin{proof}
Assume, for some ball $B\subset \Z$, that the image of $\eta$-a.e.\ curve $\gamma$ lies inside $B$. Let us consider, by Lemma \ref{lem:approx muinf}, the probability measures $(\rho_{0,n})_n ,(\rho_{1,n})_n $ converging $L^q$-strong to $ \frac{\d (\e_0)_\sharp \eta}{\d \mm_\infty}, \frac{\d (\e_1)_\sharp \eta}{\d \mm_\infty}$  with supports uniformly contained in $2B\subset \Z$. Consider $\pi_n \in \OptGeo_q(\rho_{0,n}\mm_n,\rho_{1,n}\mm_n)$ given by Theorem \ref{thm:rajalaCDqKN},  thus satisfying 
\[
\limsup_{n\uparrow\infty}\Comp(\pi_n)\le\limsup_{n\uparrow\infty}\|\rho_{0,n}\|_{L^\infty(\mm_n)}\vee \|\rho_{1,n}\|_{L^\infty(\mm_n)}\le \Comp(\eta).
\]
Finally, since  $\rho_{0,n},\rho_{1,n}$ are supported on uniformly bounded set, weak convergence implies $W_q$-convergence giving in turn, by optimality of $\pi_n$, that
\[
\limsup_{n\to\infty}\rmKe^{1/q}_q(\pi_n) = \lim_{n\to\infty}W_q( \rho_{0,n}\mm_n,\rho_{1,n}\mm_n ) =W_q((\sfe_0)_\sharp \eta,(\sfe_1)_\sharp\eta) \le  \rmKe^{1/q}_q(\eta),
\]
having used, lastly, that $(\sfe_0,\sfe_1)_\sharp \eta  \in {\rm Adm}((\sfe_0)_\sharp \eta,(\sfe_1)_\sharp\eta)$.
\end{proof}
We next face the general case of $K\in\R$ and further assume the validity of the strong curvature dimension condition to couple Theorem \ref{thm:rajalaCDqKN} and Remark \ref{rem:rajala improved}. In this case, we shall work out a polygonal geodesic interpolation rather than a single geodesic interpolation. We first face two preliminary lemmas that will be employed in the proof of Proposition \ref{prop:q_polygonal} below.
\begin{lemma}\label{lem:subplan is optimal}
Let $q\in[1,\infty)$ be fixed,  let $(\X,\sfd)$ be complete and separable and let $\mu_0,\mu_1 \in \PP(\X)$. If $\pi \in {\rm OptGeo}_q (\mu_0 ,\mu_1) $ and $f \in L^1(\pi)$ with $0\le f \le 1$, then it holds $\eta \in \OptGeo_q\left((\e_0)_\sharp\eta,(\e_1)_\sharp \eta \right)$, where $\eta \coloneqq  (\|f\|_{L^1(\pi)})^{-1} f \pi$. In particular, if $g\in L^1(\mu_0)$ with $0\le g \le 1$ then
\[
     \eta=\frac{g\circ\e_0}{\| g\|_{L^1(\mu_0)}} \pi,\qquad  \eta \in \OptGeo_q\left( \frac{g\mu_0}{\|g\|_{L^1(\mu_0)}}, (\e_1)_\sharp \eta\right).
\]
\end{lemma}
\begin{proof}
    The last conclusion is straightforward choosing $f\coloneqq g\circ \e_0$. We shall only prove that $\eta \in \OptGeo_q\left((\e_0)_\sharp\eta,(\e_1)_\sharp \eta \right)$. If not, there exists $ \bar \eta \in \OptGeo_q\left((\e_0)_\sharp\eta,(\e_1)_\sharp \eta \right)$ so that
    \[
        W_q\left((\e_0)_\sharp\eta,(\e_1)_\sharp \eta \right) = \| \sfd(\e_0,\e_1)\|_{L^q(\bar \eta)}< \| \sfd(\e_0,\e_1)\|_{L^q(\eta)}.
    \]
    By optimality and straightforward manipulations, we obtain
    \begin{equation}\label{eq:W contradiction}
    \begin{aligned} 
        W_q^q(\mu_0,\mu_1)  &=   \|f\|_{L^1(\pi)} \|\sfd(\e_0,\e_1)\|_{L^q(\eta)}^q + \|\sfd(\e_0,\e_1)\|_{L^q((1-f) \pi)}^q \\
        &>\|f\|_{L^1(\pi)}  \| \sfd(\e_0,\e_1)\|^q_{L^q(\bar \eta)}  + \|\sfd(\e_0,\e_1)\|_{L^q((1-f) \pi)}^q =\| \sfd(\e_0,\e_1)\|_{L^q (\bar \pi)}^q,
    \end{aligned}
    \end{equation}
    having set $\bar \pi \coloneqq \|f\|_{L^1(\pi)}\bar \eta + (1-f) \pi$. Notice, by construction, that $\bar \pi \in \PP(C([0,1],\X))$ and 
    \[
    (\e_i)_\sharp \bar \pi = \|f\|_{L^1(\pi)}(\e_i)_\sharp \bar \eta + (\e_i)_\sharp(1-f \pi) = (\e_i)_\sharp (f \pi) + (\e_i)_\sharp(1-f\pi) = (\e_i)_\sharp \pi,
    \]
    for $i=0,1$. Therefore, $(\e_0,\e_1)_\sharp \bar \pi \in {\rm Adm}(\mu_0,\mu_1)$ which yields a contradiction with \eqref{eq:W contradiction}.
\end{proof}
\begin{proposition}\label{propoptsubplan}
    Let $q\in [1,\infty),M\in\mathbb{N}$ be fixed and let $\Xdm$ be a metric measure space. Let $(\tilde{\rho}_i )_{i=0}^{M}\subset L^\infty (\mm)$ be probability measures and let $\tilde{\eta}_i \in \OptGeo_q (\tilde \rho_i \mm,\tilde \rho_{i+1}\mm)$ for $i=0,\ldots M-1$. Suppose there is $\Gamma_i \subset C([0,1],\X)$ Borel so that $\sum_{i=0}^{M-1}\tilde{\eta}_i (\Gamma_i^c)\leq \frac{1}{2}$. Then, there exist probability densities $(\rho_i )_{i=0}^{M}\subset L^\infty (\mm)$ and $\eta_i \in \OptGeo_q (\rho_i \mm,\rho_{i+1}\mm)$, with $\supp(\rho_i) \subseteq \supp (\tilde{\rho}_i)$ and $\supp (\eta_i) \subseteq \supp (\tilde{\eta}_i)$, such that
    \begin{itemize}
        \item[{\rm (a)}] $\eta_i(\Gamma_i^c)=0$ for all $i=0,\dots,M-1$;
        \item[{\rm (b)}] $\|\rho_i\|_{L^\infty(\mm)}\leq \left(\frac{1}{1-\sum_{j=0}^{M-1}\tilde \eta_j(\Gamma_j^c)}\right) \|\tilde{\rho}_i\|_{L^\infty(\mm)}$ for all $i=0,\dots M$;
        \item[{\rm (c)}]  $\|\tilde{\rho_i}-\rho_i\|_{L^1(\mm)}\leq 2\sum_{j=0}^{M-1}\tilde{\eta}_j (\Gamma_j^c)$ for all $i=0,\dots M$;
        \item[{\rm (d)}] $\rmKe_q(\eta_i) \le \left(\frac{1}{1-\sum_{j=0}^{M-1}\tilde \eta_j(\Gamma_j^c)}\right)\rmKe_q(\tilde \eta_i)$ for all $i=0,\dots,M-1$;
    \end{itemize}
\end{proposition}
\begin{proof}
We shall proceed inductively with finitely many steps $j \in \{0, 1, \ldots, M\}$ to construct $(\rho_{i}^j)_{i=0}^M \subset L^\infty (\mm)$ and Borel functions $(f_i^j)_{i=0}^{M-1}$, $f_{i}^j \colon C([0,1],\X) \to[0,1]$ such that, setting $\eta^{j}_i :=f_{i}^j \tilde{\eta}_i$, it holds
\begin{align}
    &(\e_{0})_\sharp \eta^{j}_i =\rho_{i}^j \mm,\qquad  (\e_{1})_\sharp \eta^{j}_{i} =\rho_{i+1}^j\mm;\label{iterativecorrectmarginal} \\
    &\|\rho_{i}^{j}-\rho_{i}^{j-1}\|_{L^1(\mm)}\leq \tilde{\eta}_{j-1}(\Gamma^{c}_{j-1})\qquad \textrm{for $j\geq 1$};\label{boundL1extra} \\
    &\|\rho_{i}^{j}\|_{L^1(\mm)}\geq 1-\sum_{k=0}^{j-1} \tilde{\eta_{k}}(\Gamma^{c}_{k})  \qquad \textrm{for $j\geq 1$};\label{lowerL1iterative} \\
    &f_{i}^{j}\restr{\Gamma^{c}_{i}}\equiv 0 \,\,\textrm{for}\,\, i< j\,\,\,\textrm{and}\,\,f_{i}^{j}\leq f_{i}^{j-1}.\label{poli0onbadcurves}
\end{align}
Note that \eqref{iterativecorrectmarginal} and $f_{i}^{j}\leq f_{i}^{j-1}$ for some $i,j$ yields  $\rho_{i}^j\leq \rho_{i}^{j-1}$ and $\rho_{i+1}^j\leq \rho_{i+1}^{j-1}$. We first show the construction and, at the end, exhibit $(\rho_i),(\eta_i)$ satisfying the conclusions (a),(b),(c),(d).

We set $\rho_{i}^0 :=\tilde{\rho}_i$, $\eta^{0}_i =\tilde{\eta}_i$ $f_{i}^{0}\equiv 1$.
Starting from the $(j-1)$-th step, we will now show how to produce $(\rho_{i}^{j})_{i=0}^M$, $(\eta_{i}^{j})_{i=0}^{M-1}$ and $(f_{i}^{j})_{i=0}^{M-1}$. We first set 
\[
C([0,1],\X)\ni\gamma\mapsto f_{j-1}^{j}(\gamma)\coloneqq \nchi_{\Gamma_{j-1}}(\gamma) f_{j-1}^{j-1}(\gamma) ,
\]
which is Borel and automatically defines 
\[
\eta^{j}_{j-1}\coloneqq \nchi_{\Gamma_{j-1}}\eta^{j-1}_{j-1},\qquad \rho^{j}_{j} \mm \coloneqq (\e_{1})_\sharp \eta^{j}_{j-1}.
\]
Since $\eta^{j}_{j-1}\leq \tilde \eta_{j-1}$, it is easy to check that \eqref{iterativecorrectmarginal},\eqref{boundL1extra},\eqref{lowerL1iterative} and \eqref{poli0onbadcurves} hold for $i=j-1$. In particular, $f_{j-1}^j$ takes values in $[0,1]$.

We next consider $i>j-1$ (the case $i<j-1$ will be the same, and it is omitted) and we build next $\rho_{i+1}^{j} ,\eta_{i}^{j}$ and $f_{i}^{j}$ from $\rho_i^{j} ,\eta_i^{j-1}$ and $\rho_i^{j-1}$ (recall, the $(j-1)$-step is assumed, hence these are available for all $i$).
We define
\[
C([0,1],\X)\ni\gamma\mapsto  f^{j}_{i}(\gamma)\coloneqq \left(\frac{\rho^{j}_i}{\rho^{j-1}_i}\chi_{\{\rho^{j-1}_i>0\}}\right)\circ \e_0(\gamma),
\]
which is Borel and, as before, it automatically defines
\[
\eta^{j}_{i}\coloneqq f^{j}_{i}\eta_{i}^{j-1},\qquad \rho_{i+1}^{j}\mm \coloneqq (\e_1)_\sharp\eta^{j}_{i}.
\]
The properties \eqref{iterativecorrectmarginal},\eqref{poli0onbadcurves} hold by construction. In particular, $f_{i}^j$ takes values in $[0,1]$. Instead, \eqref{boundL1extra} follows noticing that for 
$i>j-1$ 
\begin{equation}
\begin{split}
    \|\rho_{i}^{j-1}-\rho_{i}^{j}\|_{L^1(\mm)}&=|\e_{1\#}(\eta_{i-1}^{j-1}-\eta_{i-1}^{j})|(C([0,1],\X))= | \eta_{i-1}^{j-1}-\eta_{i-1}^{j}|(C([0,1],\X))\\
    &=|\e_{0\#}(\eta_{i-1}^{j-1}-\eta_{i-1}^{j})|(C([0,1],\X))=\|\rho_{i-1}^{j-1}-\rho_{i-1}^{j}\|_{L^1(\mm)},
    \end{split}
\end{equation}
which implies
\[
\|\rho_{i}^{j-1}-\rho_{i}^{j}\|_{L^1(\mm)}= | \eta_{j-1}^{j-1}-\eta_{j-1}^{j}|(C([0,1],\X))\leq \tilde{\eta}_{j-1}(\Gamma^{c}_{j-1}).
\]
Finally, property \eqref{lowerL1iterative} easily follows from \eqref{boundL1extra} since 
\[
\|\rho_{i}^{j}\|_{L^1(\mm)}=\left\|\rho_{i}^{0}-\sum_{k=0}^{j-1}(\rho_{i}^{k}-\rho_{i}^{k+1})\right\|_{L^1(\mm)}\geq \|\rho_{i}^{0}\|_1-\sum_{k=0}^{j-1}\|\rho_{i}^{k}-\rho_{i}^{k+1}\|_{L^1(\mm)} \geq 1-\sum_{k=0}^{j-1} \tilde{\eta}_{k}(\Gamma^{c}_{k}) .
\]
As said, the case $i<j-1$ can be done similarly. We are thus ready to define the $(\rho_i )_{i}$, $(\eta)_{i}$ and prove all the listed properties to conclude the proof. For $i=0,\ldots M$, we set
\[
\rho_{i}=\frac{\rho_{i}^M}{\|\rho_{i}^M\|_{L^1(\mm)}},\qquad\eta_{i}=\frac{\eta_{i}^M}{\eta_{i}^M(C([0,1],\X)) }.
\]
Observe that, by construction, $\supp(\rho_i) \subseteq \supp (\tilde{\rho}_i)$ and $\supp (\eta_i) \subseteq \supp(\tilde{\eta}_i)$. 
By \eqref{propoptsubplan} and Lemma \ref{lem:subplan is optimal}, we have $\eta_i \in \OptGeo_q (\rho_i \mm,\rho_{i+1}\mm)$. By construction, the conclusion (a) is obvious. Since  $\rho_{i}^M \leq \tilde{\rho_{i}}$ and using \eqref{lowerL1iterative}, we also deduce (b). The conclusion (c) is instead implied by 
\[
\|\tilde{\rho_i}-\rho_i\|_{L^1(\mm)}\leq\sum_{j=1}^M\|\rho^{j}_i-\rho_{i}^{j-1}\|_{L^1(\mm)} + \|\rho_{i}^M-\rho_{i}\|_{L^1(\mm)}\leq \sum_{j=0}^{M-1}\tilde{\eta}_j (\Gamma^{c}_j) + 1-\|\rho_{i}^M\|_{L^1(\mm)} \leq 2\sum_{j=0}^{M-1}\tilde{\eta}_j (\Gamma^{c}_j),
\]
where we exploited properties \eqref{boundL1extra},\eqref{lowerL1iterative}. Finally, the conclusion (d) follows as by construction we have
\[
    \rmKe_q(\eta_i) = \frac{\int\sfd^q(\gamma_0,\gamma_1)\,\d\eta_i^M}{\eta_{i}^M(C([0,1],\X))} \le \frac{\int\sfd^q(\gamma_0,\gamma_1)\,\d\tilde \eta_i}{\|\rho_i^M\|_{L^1(\mm)}} \overset{\eqref{lowerL1iterative}}{\le} \left(\frac{1}{1-\sum_{j=0}^{M-1}\tilde \eta_j(\Gamma_j^c)}\right)\rmKe_q(\tilde{\eta}_i),
\]
having used that $\eta_i^M \le \tilde \eta_i$ by recursion and since $f_i^j \le 1$ for all $j$.
\end{proof}
\begin{proposition}[Polygonal $W_q$-geodesics on varying spaces]\label{prop:q_polygonal}
Let $q\in(1,\infty)$ be fixed. Let $\Xdmxn$ be a sequence of pointed metric measure spaces satisfying strong ${\sf CD}_q(K,\infty)$ for some $K\in \R$. Assume that $\X_n \overset{pmG}{\rightarrow} \X_\infty$ for some $\Xdmxinf$. 

Then, for every $q$-test plan $\eta  \in \PP(C([0,1],\X_\infty))$ with bounded support, there is $M_0 \in \N$ and, for all $M\ge M_0$, there is  $n_0(M) \in \N$ and  there are $q$-test plans $\pi_{n}^M   \in \PP(C([0,1],\X_{n}))$ for all $n\ge n_0(M)$ so that the following holds:
\begin{itemize}
\item[(i)] $\pi_{n}^M$ is an $M$-polygonal geodesic (with bounded support), that is
\[
\left({\sf rest}_{\frac{i}{M}}^{\frac{i+1}{M}} \right)_\sharp \pi^M_{n} \in \OptGeo_q\left((\e_{\frac iM})_\sharp \pi_{n}^M, (\e_{\frac{i+1}M})_\sharp \pi_{n}^M\right),\qquad \forall i=0,1,...,M-1 .\]
%
\item[(ii)] $\limsup_{M\uparrow\infty}\limsup_{n \uparrow \infty}\rmKe_q(\pi_{n}^M) \le \rmKe_q(\eta)$; 
\item[(iii)] $\limsup_{M\uparrow\infty}\limsup_{n \uparrow\infty} \Comp(\pi^M_{n}) \le \Comp(\eta);$ 
\item [(iv)] there are boundedly supported probability densities $\xi_n,\zeta_n \in L^\infty(\mm_n)$ so that
\begin{align*}
    &\xi_n \to \frac{\d (\e_0)_\sharp \eta }{\d \mm_{\infty}}\quad \text{in $L^q$-strong},& & \zeta_n \to \frac{\d (\e_1)_\sharp \eta }{\d \mm_{\infty}} \quad \text{in $L^q$-strong},\\
    &\limsup_{M\uparrow\infty}\limsup_{n\uparrow\infty}\left\| \xi_n - \frac{\d (\e_0)_\sharp \pi^{M}_{n}}{\d \mm_n}\right\|_{L^q(\mm_n)} =0,&
    &\limsup_{M\uparrow\infty}\limsup_{n\uparrow\infty}\left\| \zeta_n - \frac{\d (\e_1)_\sharp \pi^{M}_{n}}{\d \mm_n} \right\|_{L^q(\mm_n)}=0.
\end{align*}
\end{itemize}
 Furthermore, if $(\Z,\sfd)$ is the realization of the convergence, $A\subset \Z$ is open, and ${\rm dist}([\eta],A^c)>0 $, then we can also require that ${\rm dist}([\pi_n^M],A^c)>0$ up to increasing further $M_0,n_0(M)$ depending also on $A$.
\end{proposition}
\begin{proof}
We subdivide the proof into different steps.

\noindent\textbf{Preliminary observations}. To prove the first part of the statement, we can replace $K$ with $K^-$ in what follows. Since $\eta$ is of bounded support we can assume that $\Gamma\coloneqq \{\gamma_t \colon t \in [0,1],\gamma \in \supp(\eta)\}$ is contained in an open ball $B\subset \Z$ with $D\coloneqq$ diam$(B)<\infty$ and $\dist(\Gamma,\Z\setminus B)>0$. Then we consider fixed $\nchi \colon \Z \to [0,1]$ a $1$-Lipschitz cut-off with $\nchi\equiv 1 $ on $B$ and with support in $2B$. It will be clear during the proof that all the probability measures involved are supported in $2B$. In particular, in what follows, weak and Wasserstein convergence will always coincide. 

\noindent\textbf{Construction of the polygonal I}. Let us consider $M\in\N,M\ge M_0$ arbitrary for $M_0$ to be chosen in the sequel. For every $i=0,1,...,M$, we can invoke Lemma \ref{lem:approx muinf} and find an integer $n_0\coloneqq n_0(M)$ depending on $M$ and, for all $n\ge n_0$, probability measures $\tilde \rho_{i,n}\mm_n \in \PP(\X_n)$ satisfying
\begin{subequations}\begin{align}
\label{eq:Wqconvergence_i}
&\tilde \rho_{i,n} \to \rho_i\coloneqq \frac{\d (\e_{t_{i }})_\sharp \eta}{\d\mm_\infty} \qquad \text{in } L^q\text{-strong as }n\uparrow\infty, \\
\label{eq:limsdens_i}
&\limsup_{n\uparrow\infty} \|\tilde \rho_{i,n}\|_{L^\infty(\mm_n)}\le \| \rho_i\|_{L^\infty(\mm_\infty)}\le {\rm Comp}(\eta),
\end{align}
\end{subequations}
 and with supports contained in $2B\subset \Z$. Thus, $W_q$-convergence also occurs.  Here, $t_{i}\coloneqq\frac{i}{M}$ is the uniform grid in $[0,1]$. Now, for all $i=0,...,M-1$, consider $ \tilde \eta_{i,n} \in \OptGeo_q(\tilde \rho_{i,n}\mm_n,\tilde\rho_{i+1,n}\mm_n)$ and define the Borel set
\begin{equation}\label{eq: gamma_i,n}
\Gamma_{i,n}\coloneqq \Big\{ \gamma \in C([0,1],\X_n) \colon  \sfd^q(\gamma_0,\gamma_1) \leq W_q(\tilde \rho_{i,n}\mm_n,\tilde\rho_{i+1,n}\mm_n)^{\frac{q-1}{2}} \Big\}.
\end{equation}
An application of Chebyshev and H\"older inequalities gives
\begin{equation}\label{Chebyshev}
\tilde  \eta_{i,n}(\Gamma_{i,n}^c) \le \frac{\int \sfd^q(\gamma_0,\gamma_1)\,\d  \eta_{i,n}}{W_q(\tilde \rho_{i,n}\mm_n,\tilde\rho_{i+1,n}\mm_n)^{\frac{q-1}{2}}} \le  W_q(\tilde \rho_{i,n}\mm_n,\tilde\rho_{i+1,n}\mm_n)^{\frac{q+1}{2}}.
\end{equation}

\noindent\textbf{Construction of the polygonal II}. In this step, we will build a sequence of $M$-polygonal geodesic between carefully selected sub-marginals $\rho_{i,n} \le \tilde \rho_{i,n}$ making sure that, interpolating geodesics from $\rho_{i,n}$ to $\rho_{i+1,n}$ lie within the set of curves $\Gamma_{i,n}$. 

Let us define $\Sigma_n^M \coloneqq \sum_{i=0}^{M-1}\tilde \eta_{i,n}(\Gamma_{i,n}^c)$. Since, by Jensen's inequality, it holds
\begin{equation}
  \sum_{i=0}^{M-1} W^q_q((\e_{t_{i}})_\sharp \eta, (\e_{ t_{i+1}})_\sharp\eta) \le \sum_{i=0}^{M-1} \int\Big(\int_{ t_{i}}^{ t_{i+1}} |\dot \gamma_t|\,\d t\Big)^q\,\d \eta \le   M^{1-q}\,\rmKe_q(\eta),
\label{eq:kinetic grid points}
\end{equation}
then, using (\ref{Chebyshev}), the concavity of $t \mapsto t^\frac{q+1}{2q}$ and (\ref{eq:Wqconvergence_i}), we obtain
\begin{equation}\label{eq:discarded mass}
\begin{split}
\limsup_{n\uparrow \infty}\Sigma_n^M &\le \varlimsup_{n \up \infty}\sum_{i=0}^{M-1} W_q(\tilde \rho_{i,n}\mm_n,\tilde\rho_{i+1,n}\mm_n)^{\frac{q+1}{2}}=\sum_{i=0}^{M-1} \left(W^q_q((\e_{t_{i}})_\sharp \eta, (\e_{ t_{i+1}})_\sharp\eta) \right)^\frac{q+1}{2q} \\
&\le M\left(\frac1M M^{1-q}\,\rmKe_q(\eta)\right)^\frac{q+1}{2q}=\frac{\rmKe_q(\eta)^\frac{q+1}{2q}}{M^{\frac{q+1}{2}-1}}.
\end{split}
\end{equation}
Note that $\frac{q+1}{2}>1$. We can now choose $M_0 \in \N$ so that  $\limsup_{n\uparrow \infty} \Sigma_n^M\le \frac 14$ for all $M\ge M_0$, and we can further increase $n_0(M)\in \N$ so that $\Sigma_n^M\le \frac 12$ for all $n\ge n_0$. Thanks to these choices, for all $M\ge M_0$ and $n\ge n_0(M)$, we are in position to invoke Proposition \ref{propoptsubplan} to obtain the submarginals $\rho_{i,n}\le \tilde \rho_{i,n}$ and the sub-plans $\eta_{i,n}\le \tilde \eta_{i,n}$ satisfying $\eta_{i,n} \in \OptGeo_q(\rho_{i,n}\mm_n,\rho_{i+1,n}\mm_n)$ and all the listed conclusions (a),(b),(c),(d) which we are going to use next. Moreover, exploiting the strong curvature dimension condition, we deduce coupling Theorem \ref{thm:rajalaCDqKN} and Remark \ref{rem:rajala improved} the following crucial property
\begin{equation}\label{eq:comp q polygonal}
    \Comp(\eta_{i,n}) \overset{\eqref{eq:rajala improved}}{\le} e^{\frac{K^-}{12} \Lip(\eta_{i,n})^2} \bigvee_{j=0}^{M}\|\rho_{j,n}\|_{L^\infty(\mm_n)} \le  e^{\frac{K^-}{12} W_q(\tilde{\rho}_{i,n}\mm_n,\tilde{\rho}_{i+1,n}\mm_n )^{\frac{q-1}{q}}}\bigvee_{j=0}^{M}\|\rho_{j,n}\|_{L^\infty(\mm_n)},
\end{equation}
for all $i=0,\dots,M-1$, having used that $\sfd(\gamma_0,\gamma_1) \le W_q(\tilde{\rho}_{i,n}\mm_n,\tilde{\rho}_{i+1,n}\mm_n )^{\frac{q-1}{2q}}$ for $\eta_{i,n}$-a.e.\ $\gamma \in \Gamma_{i,n}$ and that $\eta_{i,n}(\Gamma_{i,n}^c)=0$, by (a) in Proposition \ref{propoptsubplan}. Lastly, with a gluing argument (see, e.g., \cite[Lemma 2.1.1]{G11}), we can find $\pi_n^M \in \PP(C([0,1],\X_n))$ with the property
\begin{equation}\label{eq:piecewise geodesics}
   \eta_{i,n}=\left({\sf rest}_{t_{i}}^{t_{i+1}} \right)_\sharp \pi^M_{n},\qquad \forall i=0,1,...,M-1.
\end{equation}

\noindent\textbf{Proof of properties} (i),(ii),(iii).
It is clear from the construction that $\pi^M_{n}$ are all concentrated on piece-wise geodesics living in the bounded region $2B\subset \Z$. In particular, taking also into account \eqref{eq:comp q polygonal}, it is already evident that  $\pi^M_{n}$ is a $q$-test plan. This fact, however, will also follow by (ii)-(iii). \\
Property (i) is obvious from \eqref{eq:piecewise geodesics}. We now prove  (iii). For  $M\ge M_0$ we can estimate
\[
\limsup_{n\uparrow\infty} \Comp(\pi^M_{n}) \overset{\eqref{eq:comp q polygonal}}{\le}  e^{\frac{K^-}{12} \vee_i W_q((\e_{t_{i}})_\sharp \eta, (\e_{ t_{i+1}})_\sharp\eta) )^{\frac{q-1}{q}}}\limsup_{n\uparrow\infty}\bigvee_{i=0}^{M}\|\rho_{i,n}\|_{L^\infty(\mm_n)},
\]
having used that $\tilde \rho_{i,n}\mm_n$ converge weakly to $(\e_{t_{i}})_\sharp \eta$. As $\lim_{M \up \infty}\vee_i W_q((\e_{t_{i}})_\sharp \eta, (\e_{ t_{i+1}})_\sharp\eta) )=0$ thanks to (\ref{eq:kinetic grid points}), property (iii) follows taking also into account (b) in Proposition \ref{propoptsubplan}, then \eqref{eq:discarded mass} and lastly \eqref{eq:limsdens_i}. Let us prove now (ii). We start estimating the kinetic energy as follows:
\begin{align*}
\rmKe_q(\pi^M_{n}) &= \sum_{i=0}^{M-1}  \iint_{\frac iM}^{\frac{i+1}M}|\dot\gamma_t|^q\,\d t \d \pi^M_{n}\\
&=\sum_{i=0}^{M-1}  \iint_0^1M^{q-1}|\dot\gamma_t|^q\,\d t \d \big({\sf rest}_{\frac iM}^{\frac{i+1}M}\big)_\sharp \pi^M_{n} =\sum_{i=0}^{M-1} M^{q-1}\rmKe_q(\eta_{i,n}) \\
&{\le} \sum_{i=0}^{M-1} \frac{M^{q-1}}{1-\Sigma_n^M}\rmKe_q(\tilde \eta_{i,n}) =\frac{ M^{q-1}}{1-\Sigma_n^M}\sum_{i=0}^{M-1} W^q_q(\tilde \rho_{i,n}\mm_n,\tilde \rho_{i+1,n}\mm_n),
\end{align*}
having used, in the third line, property (d) in Proposition \ref{propoptsubplan}. Therefore, sending first $n$ and then $M$ to infinity, we reach
\[
\begin{aligned}
\limsup_{M\uparrow\infty}\limsup_{n\uparrow\infty}\rmKe_q(\pi^M_{n})& \le \left(\limsup_{M\uparrow\infty } \limsup_{n\uparrow\infty} \frac{1}{1-\Sigma_n^M}\right)\limsup_{M\uparrow\infty}M^{q-1}\sum_{i=0}^{M-1}  W_q^q((\e_{t_{i}})_\sharp \eta, (\e_{t_{i+1}})_\sharp\eta) \\
&\overset{\eqref{eq:discarded mass}}{\le}  \limsup_{M\uparrow\infty} M^{q-1}\sum_{i=0}^{M-1}W^q_q((\e_{t_{i}})_\sharp \eta, (\e_{t_{i+1}})_\sharp\eta) \overset{\eqref{eq:kinetic grid points}}{\le} \rmKe_q(\eta).
\end{aligned}
\]

\noindent\textbf{Proof of property} (iv). For each $j\in \{0,M\}$, we estimate by conclusion (c) in Proposition \ref{propoptsubplan}
\begin{align*}
    \| \rho_{j,n}-\tilde \rho_{j,n}\|^q_{L^q(\mm_n)} &\le \left(\|\rho_{j,n}\|_{L^\infty(\mm_n)} + \|\tilde \rho_{j,n}\|_{L^\infty(\mm_n)} \right)^{q-1}\| \rho_{j,n}-\tilde \rho_{j,n}\|_{L^1(\mm_n)} \\
    &\le \left(\|\rho_{j,n}\|_{L^\infty(\mm_n)} + \|\tilde \rho_{j,n}\|_{L^\infty(\mm_n)} \right)^{q-1}2\Sigma_n^M.
\end{align*}
Thanks to conclusion (b) in Proposition \ref{propoptsubplan} and recalling \eqref{eq:limsdens_i}, we get that 
\[
\limsup_{n\uparrow\infty} \| \rho_{j,n}-\tilde \rho_{j,n}\|^q_{L^q(\mm_n)} \le c(\eta)^{q-1} \limsup_{n\uparrow\infty}\Sigma_n^M,
\]
for some constant $c(\eta)>0$ depending only on $\Comp(\eta)$. Notice that the right-hand side converges to zero a $M\uparrow \infty$ by \eqref{eq:discarded mass}. The proof is then concluded with the choice $\xi_n = \tilde{\rho}_{0,n},\zeta_n= \tilde{\rho}_{M,n}$ (the latter does not depend on $M$). 

\noindent\textbf{Last conclusion}.
We conclude the proof by discussing how to modify the construction for $A\subset \Z$ open. By applying Lemma \ref{lem:approx muinf} with $\eps<\frac{1}{2}{\rm dist}([\eta],A^c)$ and $B=[\eta]$ when we first introduce the $\tilde{\rho}_{i,n}$, $i=0,\ldots,M$ and $n\in \N$, we may assume that 
$${\rm dist}({\rm supp} \, \rho_{i,n},A^c)\geq{\rm dist}({\rm supp} \, \tilde{\rho}_{i,n},A^c)>\frac{1}{2}{\rm dist}([\eta],A^c).$$
Recalling (\ref{eq: gamma_i,n}) and (\ref{eq:kinetic grid points}), we can choose $M_0$ with the additional requirement that, for any $M \ge M_0$,
$$\sup_{n \geq n_0(M)}\max_{i\in\{0,\ldots M-1\}} \Lip(\eta_{i,n})\leq \sup_{n \geq n_0(M)}\max_{i\in\{0,\ldots M-1\}} W_q(\tilde \rho_{i,n}\mm_n,\tilde\rho_{i+1,n}\mm_n)^{\frac{q-1}{2q}} < \frac{1}{4}{\rm dist}([\eta],A^c).$$
Thus we finally infer that ${\rm dist}([\pi_n ^M],A^c)>\frac{1}{4}{\rm dist}([\eta],A^c)>0$ for all $M \ge M_0$ and all $n\geq n_0(M)$.

\end{proof} 
\subsection{The case $q=\infty$}
We shall prove an $\infty$-version of the polygonal interpolation on varying spaces in Proposition \ref{prop:infty_polgeo} below.

We start with a preliminary step that is needed to find $W_\infty$-geodesics with good compression estimates. The main argument has already been carried out in \cite[Theorem 3.2]{NobiliPasqualettoSchultz22}. For the sake of completeness, and to consider more general infinite-dimensional settings, we included the proof.
\begin{proposition}\label{prop:good infty plan}
Let us assume that, for a given sequence  $q_k\uparrow \infty$,  $\Xdm$ is a metric measure space satisfying strong ${\sf CD}_{q_k}(K,\infty)$ for some $K \in \R,N\in[1,\infty)$ and for all $k\in\N$. Then, for every boundedly supported probability densities $\rho_0,\rho_1\in L^\infty(\mm)$, there is $\pi \in \OptGeo_\infty(\rho_0\mm,\rho_1\mm)$ so that
\[
\Comp(\pi) \le e^{\frac{K^-}{12}W^2_\infty(\rho_0\mm,\rho_1\mm)}\|\rho\|_{L^\infty(\mm)}\vee \|\rho_1\|_{L^\infty(\mm)}.
\]
In particular, $\pi$ is an $\infty$-test plan. If $K\ge 0$, the same  also holds if $\Xdm$ is assumed $\CD_{q_k}(K,\infty)$.
\end{proposition}
\begin{proof}
For brevity, we set $\mu_i=\rho_i\mm$, for $i=0,1$. To prove the first part of the statement, we can replace $K$ with $K^-$ in what follows. In the end, we discuss the case $K\ge 0$. For each $k \in \N$, consider $\eta_k \in \OptGeo_{q_k}(\mu_0,\mu_1)$, define $\eps_k \coloneqq W_\infty(\mu_0,\mu_1)(q_k^{1/q_k}-1)$ and set
\[
\Gamma_k\coloneqq \Big\{ \gamma \in C([0,1],\X) \colon \int_0^1 |\dot \gamma_t|^{q_k}\,\d t \le (W_\infty(\mu_0,\mu_1)+\eps_k)^{q_k}\Big\},\qquad  \pi_k \coloneqq \frac{\eta_k\restr{\Gamma_k}}{\eta_k(\Gamma_k)},
\]
whenever $\eta_k(\Gamma_k)>0$. By Chebyshev's, we have 
\[
 \pi_k(\Gamma^c_k)\le \frac{W^{q_k}_{q_k}(\mu_0,\mu_1)}{q_kW^{q_k}_{\infty}(\mu_0,\mu_1)} \le \frac1{q_k}\to 0,\qquad \text{ as }k\uparrow \infty,
\]
so that,  eventually, $ \pi_k$ is well-defined for all $k$ large enough. Being a restriction, $ \pi_k$ is optimal with respect its own endpoints, i.e.\ $ \pi_k \in \OptGeo_{q_k}((\e_0)_\sharp \pi_k,(\e_1)_\sharp  \pi_k)$ (cf.\ Lemma \ref{lem:subplan is optimal}). Moreover, by Remark \ref{rem:rajala improved}, it is the unique dynamical plan satisfying
\begin{equation}
\Comp(\pi_k)  \le \frac{e^{\frac{K^-}{12}(W_\infty\left(\mu_0,\mu_1)+\eps_k\right)^2 }}{\eta(\Gamma_k)} \|\rho_0\|_{L^\infty(\mm)}\vee \|\rho_1\|_{L^\infty(\mm)}
\label{eq:stima curvewise}
\end{equation}
for all $t\in[0,1]$, having used that $\Lip( \pi_k)\le (W_\infty(\mu_0,\mu_1)+\eps_k)$ and that the endpoints of $\pi_k$ satisfies the uniform estimates
\[
 (\e_i)_\sharp \pi_k  \le \frac{(\e_i)_\sharp \eta_k}{\eta(\Gamma_k)} 
  \le \frac 1{\eta(\Gamma_k)}\|\rho_0\|_{L^\infty(\mm)}\vee \|\rho_1\|_{L^\infty(\mm)} \mm,\qquad \forall i=0,1.
\]
Next, notice that $\sup_k \Comp(\pi_k)<\infty$, $\sup_k\Lip(\pi_k)<\infty$ (as they are concentrated on equi-Lipschitz geodesics), and the zero-marginal is fixed with bounded support. Hence, $(\pi_k)$ is a sequence of $\infty$-test plans satisfying the assumptions of Proposition \ref{prop:infty_tightness} and, up to a not relabeled subsequence, $\pi_k \weakto  \pi$ for some $\infty$-test plan  $ \pi$ as $k\uparrow\infty$.

We claim that $\pi$ does the job. First, notice that $(\e_i)_\sharp \pi = \mu_i$, for $i=0,1$. This can be seen as follows: notice that $(\e_i)_\sharp \pi_k = \eta_k(\Gamma_k)^{-1}(\e_i)_\sharp \eta_k\restr{\Gamma_k} \weakto \mu_i$ as $k\uparrow\infty$ in duality with $C_b(\X)$, for $i=0,1$. Hence, the claim follows from weak continuity of pushforward via the evaluation map. Also, by \eqref{eq:Mosco Keq Lip}, we have
\begin{align*}
\Lip(\pi) &\le \liminf_{k\to\infty}\rmKe_{q_k}^{1/q_k}(\pi_k) = \liminf_{k\to\infty}  W_{q_k}^{1/q_k}((\e_0)_\sharp \pi_k,(\e_1)_\sharp  \pi_k) = \liminf_{k\to\infty}  \rmKe_{q_k}^{1/q_k}( \pi_k) \\
&\le \limsup_{k\to\infty}\frac{1}{\eta(\Gamma_k)}\rmKe_{q_k}^{1/q_k}(\eta_k)  = \lim_{k\to\infty}W_{q_k}(\mu_0,\mu_1) = W_{\infty}(\mu_0,\mu_1).
\end{align*}
Thus \eqref{eq:dynOPinfty} holds and $\pi \in \OptGeo_\infty(\mu_0,\mu_1)$. Since compression bounds are stable under weak convergence, it holds
\[
\Comp(\pi)\le \liminf_{k\to\infty}\Comp(\pi_k) \le\limsup_{k\uparrow\infty} \frac{e^{\frac{K^-}{12}(W_\infty(\mu_0,\mu_1)+\eps_k)^2}}{\eta(\Gamma_k)} \|\rho_0\|_{L^\infty(\mm)}\vee \|\rho_1\|_{L^\infty(\mm)}.
\]
Finally, if $K\ge 0$ and $\Xdm$ is a $\CD_{q_k}(K,\infty)$ space, we see that the application of Theorem \ref{thm:rajalaCDqKN} still guarantees the validity of \eqref{eq:stima curvewise} when $K =0$, even though $\pi_k$ might be not unique. Hence, all the rest of the argument still applies.
\end{proof}
The following lemma, playing the same roles as Lemma \ref{lem:approx muinf} and Lemma \ref{lem:subplan is optimal}, will be needed to build polygonal geodesic interpolations in the $W_\infty$-topology.

\begin{lemma}\label{lemma_approxWinf}
Let $\Xdmxn$ be pointed metric measure spaces satisfying $\X_n \overset{pmG}{\rightarrow} \X_\infty$ for some $\Xdmxinf$, and fix a realization $(\Z,\sfd)$. Let $M\in\N$ and let $(\rho_i)_{i=0}^M \subset L^\infty(\mm_\infty)$ be probability densities with $\supp(\rho_i) \subset B$ for all $i=0, \ldots, M$ and for some  bounded set $B\subset \Z$.  Fix $\eps>0$. Then, for all $n\in\N$ there are probability densities $(\rho_{i,n})_{i=0}^M \subset L^\infty(\mm_n)$ with $\supp(\rho_{i,n}) \subset I_\eps(B)$ such that:
\begin{enumerate}[\rm(\rm a\rm)]
    \item $\varlimsup_{n \uparrow \infty} W_\infty(\rho_{i,n} \mm_n, \rho_{i+1,n} \mm_n) \leq W_\infty(\rho_i \mm_\infty, \rho_{i+1} \mm_\infty) $ for all $i=0,\ldots, M-1$; \label{Winf-est}
    \item $\varlimsup_{n \up\infty} \| \rho_{i,n}\|_{L^\infty(\mm_n)} \leq \|\rho_i\|_{L^\infty(\mm_\infty)}$ for all $i=0,\ldots,M$; \label{L^infty-est}
    \item $\rho_{i,n}$ converges in $L^q$-strong to $\rho_i$ as $n\uparrow\infty$ for all $q \in [1, \infty)$ and $i=0,\ldots,M$. \label{Lq strong-est}
\end{enumerate}
\end{lemma}
\begin{proof}
We subdivide the proof into different steps.

\noindent\textbf{Reduction to $W_2$-converging reference measures}.
Without loss of generality, we can assume that $\mm_n,\mm_\infty$ are probability measures supported in $ I_\eps(B)$, so in particular they satisfy $W_q(\mm_n, \mm_\infty) \rightarrow 0$ for all $q \in [1, \infty)$ (though it is generally false that $W_\infty(\mm_n, \mm_\infty) \rightarrow 0$). Assume, indeed, that the conclusion is valid in this case and let us show how to prove the general claim. Consider $\eta=\nchi_{ I_{\eps/2}(B)}$ with $0<\eps<1$ such that $\mm_\infty(\partial ( I_{\eps/2}(B)))=0$ (which is the case for a.e. $\eps$), so that by standard arguments we have $\lim_n \int_{ I_{\eps/2}(B)} g \,\d\mm_n=\int_{ I_{\eps/2}(B)} g \,\d\mm_\infty$ for all $g \in C_b(\Z)$. Then if we replace $\mm_n,\mm_\infty$ respectively with $\mm_n'\coloneqq \|\eta\|_{L^1(\mm_n)}^{-1}\eta\, \mm_n,\mm_\infty'\coloneqq \|\eta\|_{L^1(\mm_\infty)}^{-1}\eta \, \mm_\infty$ we still have $\mm_n' \weakto \mm_\infty'$. Therefore, if we define
\[
\rho'_{i} \coloneqq \frac{\rho_{i}}{\|\rho_i\|_{L^1(\mm_\infty')}},\qquad \text{so that}\qquad \rho'_i\mm_\infty' \in \PP(\X_\infty),
\]
the conclusion of the Lemma provides the existence of $\rho'_{i,n} \mm_n' \in \PP(\X_n)$ (in particular we can assume $\rho'_{i,n}=0$ outside $ I_{\eps/2}(B)$) satisfying all the listed properties \eqref{Winf-est},\eqref{L^infty-est},\eqref{Lq strong-est}. Thus, setting
\[
    \rho_{i,n}\mm_n\coloneqq \frac{\rho_{i,n}'\mm_n}{\|\rho_{i,n}'\|_{L^1(\mm_n)}}   \in \PP(\X_n),
\]
we get
\[
    \rho_{i,n}\mm_n = \frac{\rho_{i,n}'\|\eta\|_{L^1(\mm_n)} \eta\, \mm_n}{\|\rho_{i,n}'\|_{L^1(\mm_n)}\|\eta\|_{L^1(\mm_n)}} = \frac{\rho_{i,n}' \mm_n'}{\|\rho_{i,n}'\|_{L^1(\mm_n')}} = \rho_{i,n}'\mm_n' . 
\]
Thanks to this identity, we see that the conclusions \eqref{Winf-est},\eqref{L^infty-est},\eqref{Lq strong-est} for $(\rho_{i,n})$ follow straightforwardly from those of $\rho_{i,n}'$. Thus, we can assume and will assume in the sequel that
\[\eps_n:=\sqrt{W_2(\mm_n, \mm_\infty)} \to 0.\] 
\noindent\textbf{Construction of $(\rho_{i,n})$ and proof of \eqref{Winf-est},\eqref{L^infty-est},\eqref{Lq strong-est}}. Consider $\alpha^{(n)} \in {\rm Opt}_2(\mm_\infty, \mm_n)$ and its disintegration $\alpha^{(n)}=\mm_\infty \otimes\{ \alpha^{(n)}_x\}_x$ with respect to the first marginal (see, e.g., \cite{AmbrosioBrueSemola24_Book}). If we define
\[E_n:=\{(x,y) \in \Z \times \Z: \sfd(x,y) \leq \eps_n \},\]
 then Markov's inequality yields
\begin{equation}\label{markov-gamma}
\alpha^{(n)}(E_n^c)  \leq \frac1{\eps_n^2}W_2^2(\mm_n, \mm_\infty)=\eps_n^2.
\end{equation}

Now, the rough idea of the proof can be explained as follows: we would like to build the density $\rho_{i,n}$ by moving most of the mass of $\rho_i \mm_\infty$ with the transport plan $\rho_i(x) \alpha^{(n)} \res E_n$, and spreading the remaining mass (infinitesimal as $n \uparrow \infty$) uniformly over $\mm_n$, so that it does not contribute to $W_\infty(\rho_{i,n} \mm_n, \rho_{i+1,n} \mm_\infty)$. However, in order to make this construction work, some technical refinements are needed. Namely, for $0 \leq i < M-1$ and $\eta^{i, i+1} \in {\rm Opt}_\infty(\rho_i\mm_\infty, \rho_{i+1}\mm_\infty)$, we claim the following.
\begin{claim}\label{claim1}
There are densities $(\Tilde{\rho}_i)_{i=0}^{M} \subset L^\infty_+(\mm_\infty)$ and plans $(\tilde{\eta}^{i, i+1})_{i=0}^{M-1} \subset \MM_+(\Z\times \Z)$ (depending on $n \in \N$) such that
\begin{enumerate}[\rm(1\rm)]
    \item $\Tilde{\rho}_i(x) \leq \rho_i(x) \alpha_x^{(n)}(E_n)$ $\mm_\infty$-a.e.\ $x \in \X_\infty$ for all $i=0,\dots,M$;
    \item $\| \Tilde{\rho}_0\|_{L^1(\mm_\infty)}= \ldots =\| \Tilde{\rho}_M\|_{L^1(\mm_\infty)} \geq 1-\eps_n^2 \sum_{i=0}^M \| \rho_i\|_{L^\infty(\mm_\infty)} $;
    \item $\tilde{\eta}^{i, i+1} \leq \eta^{i, i+1}$ and $\tilde{\eta}^{i, i+1} \in {\rm Adm}(\Tilde{\rho}_i \mm_\infty, \Tilde{\rho}_{i+1} \mm_\infty)$ for all $i=0,\dots M-1$.
\end{enumerate}
\end{claim}
\noindent Now we actually construct $(\rho_{i,n})_{i=0, \ldots, M}$ taking Claim \ref{claim1} for granted. Fix any $i=0, \ldots, M$ and set
\[\alpha^{(n), i}:= \Tilde{\rho}_i \mm_\infty \otimes \left\{ \frac{\alpha_x^{(n)} \res E_n}{\alpha_x^{(n)}(E_n)}\right\}+(\rho_i-\Tilde{\rho}_i)\mm_\infty \times \mm_n, \qquad \rho_{i,n} \mm_n:= P^2_\sharp(\alpha^{(n), i}),\]
where $P^1, P^2: \Z \times \Z \to \Z$ are the canonical projections on the two components. Notice that $P^1_\sharp(\alpha^{(n), i})=\rho_i \mm_\infty$ and that $P^2_\sharp(\alpha^{(n), i}) \ll \mm_n$ (hence $\rho_{i,n}$ is well-defined), indeed by point (1)
\[P^2_\sharp(\alpha^{(n), i})=\|\rho_i-\Tilde{\rho}_i\|_{L^1(\mm_\infty)}\mm_n+\int \frac{\Tilde{\rho}_i(x)}{\alpha^{(n)}_x(E_n)} \alpha^{(n)}_x(x, \cdot) \,\d\mm_\infty(x) \leq (\|\rho_i-\Tilde{\rho}_i\|_{L^1(\mm_\infty)}+ \|\rho_i\|_{L^\infty(\mm_\infty)})\mm_n.\]
The above computation also gives the bound $\| \rho_{i,n}\|_{L^\infty(\mm_n)} \leq c_n+\| \rho_i\|_{L^\infty(\mm_\infty)}$, where we denote $c_n:=\| \rho_i-\Tilde{\rho}_i\|_{L^1(\mm_\infty)}=1-\| \Tilde{\rho}_i\|_{L^1(\mm_\infty)}$ (by (1) again). By (2) we have $c_n \leq \eps_n^2 \sum_j \|\rho_j\|_{L^\infty(\mm_\infty)}$, which proves (\ref{L^infty-est}) in our main statement. \\
Let us now show the validity of (\ref{Winf-est}). Let
\begin{equation}\label{beta_n_i}
\beta^{(n), i}:=\Tilde{\rho}_i \mm_\infty \otimes \left\{ \frac{\alpha_x^{(n)} \res E_n}{\alpha_x^{(n)}(E_n)}\right\}=\alpha^{(n), i}-(\rho_i-\Tilde{\rho}_i)\mm_\infty \times \mm_n,
\end{equation}
and notice that $P^2_\sharp(\beta^{(n), i})=(\rho_{i,n}-c_n) \mm_n$.
We estimate $W_\infty(\rho_{i,n}, \rho_{i+1,n})$ from above using the transport plan
\begin{equation}\label{candidate-plan}
\beta^{(n), i+1} \circ \tilde{\eta}^{i, i+1} \circ (\beta^{(n), i})^{-1}+({\rm Id}, {\rm Id})_\sharp(c_n \mm_n) \in {\rm Adm}(\rho_{i,n} \mm_n, \rho_{i+1,n} \mm_n).
\end{equation}
Here the ``inversion" and ``composition" of transport plans are defined in the natural way. That is, we denote $\eta^{-1}:= S_\sharp\eta$, with $S: (x,y) \mapsto (y,x)$, and given $\eta_1= \mu_1 \otimes \{(\eta_1)_x \} \in {\rm Adm}(\mu_1, \mu_2)$, $\eta_2=\mu_2 \otimes \{(\eta_2)_y \} \in {\rm Adm}(\mu_2, \mu_3)$ we set $\eta_2 \circ \eta_1:= \mu_1 \otimes \{\eta_x\} \in {\rm Adm}(\mu_1, \mu_3)$, where
\[\eta_x=(\eta_2 \circ \eta_1)_x:=\int (\eta_2)_y(y, \cdot) \, \d (\eta_1)_x(x,y).\]
It follows by the construction that $\eta_2 \circ \eta_1$ is concentrated on any set of the form $\{(x,z): \exists\, y \mbox{ s.t. }(x,y) \in C_1, (y,z) \in C_2\}$, given two sets $C_1, C_2$ so that $\eta_j$ is concentrated on $C_j$ for $j=1,2$. Thus, we have
\begin{equation}\label{eq:stimaWinf}
W_\infty(\rho_{i,n} \mm_n, \rho_{i+1,n} \mm_n) \overset{\eqref{candidate-plan}}{\leq} \|\sfd(\cdot,\cdot)\|_{L^\infty(\tilde{\eta}^{i, i+1})}+2\eps_n \leq W_\infty(\rho_i \mm_\infty, \rho_{i+1} \mm_\infty)+2\eps_n,
\end{equation}
where we have also exploited the fact that $\beta^{(n), i}, \beta^{(n), i+1}$ are concentrated on $E_n$, property (3) (which also guarantees that the compositions in (\ref{candidate-plan}) are well-posed) and the optimality of $\eta^{i, i+1}$. \\
Finally, we prove (\ref{Lq strong-est}). The weak convergence $\rho_{i,n} \mm_n \weakto \rho_i \mm_\infty$ follows by the previous estimates:
\[ \begin{split} W_2^2(\rho_i \mm_\infty, \rho_{i,n} \mm_n) & \leq \int \sfd^2(x,y) \, \d\alpha^{(n), i} \leq \int  \sfd^2(x,y) \, \d\beta^{(n), i}+ c_n(\diam(2B))^2 \\
& \leq\eps_n^2 \,\left(1+\sum_j \| \rho_j\|_{L^\infty(\mm_\infty)}(\diam(2B))^2 \right) \to 0.\end{split}\]
In order to check the strong $L^q$ estimate for $q \in [1, \infty)$, notice that if we write $\alpha^{(n)}=\mm_n \otimes \{\alpha^{(n)}_y\}_y$, disintegrating with respect to the second marginal, then we have
\begin{equation}\label{dis-formula}
\rho_{i,n}(y)=\int_{E_n} \frac{\Tilde{\rho}_i(x)}{\alpha_x^{(n)}(E_n)} \, \d \alpha_y^{(n)}(x,y)+c_n,
\end{equation}
as one can infer from (\ref{beta_n_i}) by computing explicitly the coupling $ \langle \varphi, P^2_\sharp(\beta^{(n), i})  \rangle$, for arbitrary $\varphi \in C_{bs}(\Z)$. Combining (\ref{dis-formula}) with Jensen's inequality and the property (1), we get
\[\begin{split}
\int [\rho_{i,n}(y)]^q \, \d \mm_n(y) &\leq \int \int \left(\frac{\Tilde{\rho}_i(x)}{\alpha_x^{(n)}(E_n)}+c_n \right)^q \d\alpha^{(n)}_y(x,y) \, \d \mm_n(y) \\
&\leq \int (\rho_i(x)+c_n)^q \, \d\alpha^{(n)}(x,y)=\int (\rho_i(x)+c_n)^q \, \d\mm_\infty(x),
\end{split}\]
which gives (\ref{Lq strong-est}) (for $q>1$) by taking the limit $n \up \infty$. The case $q=1$ follows instead the same arguments as in the proof of Lemma \ref{lem:approx muinf}.

\noindent\textbf{Proof of Claim \ref{claim1} and conclusion}. To conclude the proof, we are only left to show Claim \ref{claim1}. We construct the densities $(\Tilde{\rho}_i)_{i \leq M}$ and the subplans $(\tilde{\eta}^{i, i+1})_{i<M}$ by an inductive procedure in $M+1$ steps, starting from $\rho_i^{(0)}:=\rho_i$ and $\eta^{i, i+1, (0)}:=\eta^{i, i+1}$, so that at each step $k=1, \ldots, M+1$ we build families $(\rho_i^{(k)})_{i=0}^M$ and $(\eta^{i, i+1, (k)})_{i=0}^{M-1}$ with the following properties:
\\
\begin{enumerate}[(i)]
    \item $\eta^{i, i+1, (k)} \leq \eta^{i, i+1, (k-1)} $ for all $ 0\le i<M$;
    \item $\eta^{i, i+1, (k)} \in {\rm Adm}(\rho_i^{(k)}\mm_\infty, \rho_{i+1}^{(k)} \mm_\infty)$ for all $ 0\le i<M$;
    \item $\| \rho_0^{(k)}\|_{L^1(\mm_\infty)}= \ldots=\| \rho_M^{(k)}\|_{L^1(\mm_\infty)}$ and $\| \rho_0^{(k)}\|_{L^1(\mm_\infty)} \geq \| \rho_0^{(k-1)}\|_{L^1(\mm_\infty)}-\eps_n^2 \|\rho_{k-1}\|_{L^\infty(\mm_\infty)}$;
    \item $\rho_i^{(k)}(x) \leq \rho_i(x)\alpha^{(n)}_x(E_n) \, \,\,\,\mm_\infty\mbox{-a.e. }x$ for all $ 0\le i\le k-1$; \\
\end{enumerate}
Hence the choices $\Tilde{\rho}_i:=\rho_i^{(M+1)}$ and $\tilde{\eta}^{i, i+1}:=\eta^{i, i+1, (M+1)}$ fulfill all the properties listed in Claim \ref{claim1}. The inductive step (from $k-1$ to $k$) proceeds as follows: we ``cut mass" from the density $\rho_{k-1}^{(k-1)}$ by setting 
\[\rho_{k-1}^{(k)}(x):=\rho_{k-1}^{(k-1)}(x) \alpha^{(n)}_x(E_n).\]
Note that $\rho_{k-1}^{(k)}(x) \leq \rho_{k-1}(x) \alpha^{(n)}_x(E_n)$, as $(\rho_{k-1}^{(h)})_{h=0}^{k-1}$ is a decreasing family (so (iv) holds for $i=k-1$), and that we have, by (\ref{markov-gamma}),
\begin{equation}\label{mass-loss}
\begin{split}
\|\rho_{k-1}^{(k)}\|_{L^1(\mm_\infty)}&=\| \rho_{k-1}^{(k-1)}\|_{L^1(\mm_\infty)}-\int \rho_{k-1}^{(k-1)} \alpha^{(n)}_x(E_n^c) \, \d\mm_\infty \\
&\geq \| \rho_0^{(k-1)}\|_{L^1(\mm_\infty)}-\eps_n^2\| \rho_{k-1}\|_{L^\infty(\mm_\infty)}.
\end{split} 
\end{equation}
Now we adjust the other densities and transport plans accordingly to the mass loss which happened to the $k$-th density, so that (ii) and (iii) continue to hold. The argument is reminiscent of that in the proof of Proposition \ref{propoptsubplan} (with the maps $P^1, P^2$ playing the roles of $\e_0$, $\e_1$). \\
So, for $i \geq k-1$, consider the disintegrations $\eta^{i, i+1, (k-1)}=\rho_i^{(k-1)}\mm_\infty \otimes \{\eta^{i, i+1, (k-1)}_x\}_x$ with respect to the first marginal. Set
\[\eta^{k-1, k, (k)}:=\rho_{k-1}^{(k)} \mm_\infty \otimes  \{\eta^{k-1, k, (k-1)}_x\}_x, \qquad \rho_k^{(k)} \mm_\infty:=P^2_\sharp(\eta^{k-1, k, (k)}) \leq \rho_k^{(k-1)} \mm_\infty,\]
and, inductively for $i>k-1$, define
\[\eta^{i, i+1, (k)}:=\rho_i^{(k)} \mm_\infty \otimes  \{\eta^{i, i+1, (k-1)}_x\}_x, \qquad \rho_{i+1}^{(k)} \mm_\infty:=P^2_\sharp(\eta^{i, i+1, (k)}) \leq \rho_{i+1}^{(k-1)} \mm_\infty.\]
Similarly, for $i<k-1$ consider the disintegrations $\eta^{i, i+1, (k-1)}=\rho_{i+1}^{(k-1)}\mm_\infty \otimes \{\eta^{i, i+1, (k-1)}_y\}_y$ with respect to the second marginal and define inductively (but this time proceeding backwards, from $i=k-2$ to $i=0$)
\[\eta^{i, i+1, (k)}:=\rho_{i+1}^{(k)} \mm_\infty \otimes  \{\eta^{i, i+1, (k-1)}_y\}_y, \qquad \rho_i^{(k)} \mm_\infty:=P^1_\sharp(\eta^{i, i+1, (k)}) \leq \rho_i^{(k-1)} \mm_\infty.\]
Observe that all the objects defined in this construction share the same mass, which we estimated in (\ref{mass-loss}), so (iii) is proved. Finally, properties (i) and (ii) follow immediately from the construction, and thus also property (iv) for $i=0, \ldots, k-2$.
\end{proof}

Next, we shall assume to work under the assumptions of Theorem \ref{thm:CDq independent} to be able to invoke Proposition \ref{prop:good infty plan}.
\begin{proposition}[Polygonal $W_\infty$-geodesics on varying spaces]\label{prop:infty_polgeo}
Let $\Xdmxn$ be $q$-essentially non-branching pointed metric measure spaces with $\mm_n(\X_n)<\infty$ for all $q\in(1,\infty),n\in\N$. Assume $\X_n$ satisfies ${\sf CD}(K,N)$ for some $K \in \R,N \in [1,\infty)$ and that $\X_n \overset{pmG}{\rightarrow} \X_\infty$ for some $\Xdmxinf$. 

Then, for every $\infty$-test plan $\eta  \in \PP(C([0,1],\X_\infty))$ with bounded support and for every $M \in \N$, there exists an $\infty$-test plan $\pi_n^M   \in \PP(C([0,1],\X_n))$ so that:
\begin{itemize}
\item[(i)] $\pi^M_n$ is an $M$-polygonal geodesic (with bounded support), that is
\[
\left({\sf rest}_{\frac{i}{M}}^{\frac{i+1}{M}} \right)_\sharp \pi^M_{n} \in \OptGeo_\infty( (\e_\frac{i}{M})_\sharp \pi^{M}_n, \left(\e_\frac{i+1}{M})_\sharp \pi^M_n  \right),\qquad \forall i=0,1,...,M-1;
\]
\item[(ii)] $\limsup_n\Lip(\pi^M_n) \le \Lip(\eta)$ for all $M \in \N$;
\item[(iii)] $\limsup_n\Comp(\pi^M_{n})\leq e^{\frac{K^-}{12}\frac{\Lip(\eta)^2}{M^2}} \Comp(\eta)$ for all $M \in \N$;
\item[(iv)]for all $M \in\N$ and $q \in [1,\infty)$, it holds
\[
    \frac{\d (\e_0)_\sharp \pi^{M}_{n} }{\d \mm_{n}} \to \frac{\d (\e_0)_\sharp \eta }{\d \mm_{\infty}}, \qquad \frac{\d (\e_1)_\sharp \pi^{M}_{n} }{\d \mm_{n}}  \to \frac{\d (\e_1)_\sharp \eta }{\d \mm_{\infty}}, \qquad \text{in $L^q$-strong as $n\uparrow\infty$}.
\]
\end{itemize}
Furthermore, if $(\Z,\sfd)$ is the realization of the convergence, $A\subset \Z$ is open, and ${\rm dist}([\eta],A^c)>0 $, then we can also require that ${\rm dist}([\pi_n^M],A^c)>0$ for all $n$ sufficiently large and $M$ sufficiently large depending also on $A$.
\end{proposition}
\begin{proof}
Fix  $M\in\N$, an $\infty$-test plan $\eta \in \PP(C([0,1],\X_\infty))$ with bounded support and a sequence $q_k \uparrow \infty$ as $k\in\N$ goes to infinity. By assumption, each $\X_n$ is a ${\sf CD}_{q_k}(K,N)$ space for all $k\in\N$ thanks to Theorem \ref{thm:CDq independent}. We now build the sequence of $(\pi^M_n)$ as required by the statement.

We can apply Lemma \ref{lemma_approxWinf} with $\rho_i\mm_\infty\coloneqq (\e_{i/M})_\sharp \eta \in \PP(\X_\infty)$ to get $ \rho_{i,n}\mm_n \in \PP(\X_n)$ satisfying
\begin{subequations}\begin{align}\label{eq:limsdens_i_inf}
&\limsup_{n\uparrow\infty}\|\rho_{i,n}\|_{L^\infty(\mm_n)}\le \|\rho_i\|_{L^\infty(\mm_\infty)}\le {\rm Comp}(\eta),\\
&\varlimsup_{n \uparrow \infty} W_\infty(\rho_{i,n} \mm_n, \rho_{i+1,n} \mm_n) \leq W_\infty(\rho_i \mm_\infty, \rho_{i+1} \mm_\infty), \label{eq:stima infinito}
\\
&\rho_{i,n} \to \rho_{i} \qquad \text{in $L^q$-strong, for all }q\in[1,\infty), \label{eq:Winftyconvergence_i_inf}
\end{align}
\end{subequations}
and $\rho_{i,n}$ are all supported in a uniformly bounded region, for all $i$ and $n\in\N$. 

Now, for every $i=0,...,M-1$, consider $\eta_{i,n} \in \OptGeo_\infty(\rho_{i,n}\mm_n,\rho_{i+1,n}\mm_n)$  given by Proposition \ref{prop:good infty plan} satisfying
\begin{equation}
\Comp(\eta_{i,n})\le e^{\frac{K^-}{12}W_\infty^2(\rho_{i,n}\mm_n,\rho_{i+1,n}\mm_n)}\|\rho_{i,n}\|_{L^\infty(\mm_n)}\vee\|\rho_{i+1,n}\|_{L^\infty(\mm_n)}.\label{eq:comp infty proof}
\end{equation}
Fixing the uniform grid $t_{i} = \frac iM$ of $[0,1]$ and thanks to a gluing argument as previously done, we can build a plan $\pi^M_n\in \PP(C([0,1],\X_n))$ satisfying
\[
    \left({\sf rest}_{t_{i}}^{t_{i+1}}\right)_\sharp \pi^M_n = \eta_{i,n},\qquad \forall i=0,\dots M-1.
\]
We claim that $\pi_n^M$ satisfies all the listed properties. The fact that it is an $\infty$-test plan follows from the very construction, and the fact that the supports are equibounded. Property (i) follows again by the gluing construction. To see (ii), we simply estimate as follows
\[ 
\begin{split}
\limsup_{n\to\infty} \Lip(\pi^M_n) &\overset{\eqref{eq:Lippolygonal}}{=} \limsup_{n\to\infty} M \bigvee_{i=0}^{M-1} \Lip \left( \left({\sf rest}_{t_{i}}^{t_{i+1}}\right)_\sharp \pi^M_n \right)  =\limsup_{n\to\infty} M \bigvee_{i=0}^{M-1} \Lip(\eta_{i,n})\\
& = M\bigvee_{i=0}^{M-1} \lim_{n\to\infty}  W_\infty(\rho_{i,n}\mm_n,\rho_{i+1,n}\mm_n) \overset{\eqref{eq:Winftyconvergence_i_inf}}{=} M \bigvee_{i=0}^{M-1} W_\infty\left(\left(\e_{t_{i}})_\sharp \eta,(\e_{t_{i+1}}\right)_\sharp \eta\right) \\
&\overset{\eqref{eq:Winf_admissible}}{\le}  \Lip(\eta).
\end{split}
\]
Finally, to prove (iii),  we estimate
\begin{align*}
\limsup_{n\uparrow\infty}\Comp(\pi^M_n) &\overset{\eqref{eq:comp infty proof}}{\le} \limsup_{n\uparrow\infty} e^{\frac{K^-}{12}\vee_{i=0}^{M-1} W_\infty^2(\rho_{i,n}\mm_n,\rho_{i+1,n}\mm_n)}\bigvee_{i=0}^M\|\rho_{i,n}\|_{L^\infty(\mm_n) } \\
&\overset{\eqref{eq:limsdens_i_inf},\eqref{eq:stima infinito}}{\le}  e^{\frac{K^-}{12} \vee_{i=0}^{M-1}  W_\infty\left((\e_{t_{i}})_\sharp \eta,(\e_{t_{i+1}})_\sharp \eta\right) ^2} \Comp(\eta) \overset{\eqref{eq:Winf_admissible}}{\le}  e^{\frac{K^-}{12}\frac{\Lip^2(\eta)}{M^2}} \Comp(\eta).
\end{align*}
Taking now $M$ to infinity, (iii) follows. Property (iv) reads just as \eqref{eq:Winftyconvergence_i_inf} for $i=\{0,M\}$.

We conclude the proof by discussing how to modify the construction for $A\subset \Z$ open. By applying Lemma \ref{lemma_approxWinf} with $\eps<\frac{1}{2}{\rm dist}([\eta],A^c)$ and $B=[\eta]$ when we first introduce the $\rho_{i,n}$, $i=0,\ldots,M$ and $n\in \N$, we may assume that 
$${\rm dist}({\rm supp}(\rho_{i,n}),A^c)>\frac{1}{2}{\rm dist}([\eta],A^c).$$
Recalling \eqref{eq:stimaWinf} and \eqref{eq:Winf_admissible}, we can find $M_0$, $n_0$ s.t.
\[\Lip(\eta_{i,n})=W_\infty(\rho_{i,n} \mm_n, \rho_{i+1,n} \mm_n) \le W_\infty\left((\e_{t_{i}})_\sharp \eta,(\e_{t_{i+1}})_\sharp \eta\right)+2\eps_n < \frac 14{\rm dist}([\eta],A^c)  \]
for all $M \ge M_0$ and $n \ge n_0$, where $\eps_n \to 0$ is the same as in \eqref{eq:stimaWinf} (note that $\eps_n$ depends only on $\X_n$, $\X_\infty$ and $[\eta]$, $A$). Thus we finally infer that ${\rm dist}([\pi_n ^M],A^c)>\frac{1}{4}{\rm dist}([\eta],A^c)>0$ for $n$, $M$ large enough.

\end{proof}
\section{Mosco-convergence results}
In this part, we prove our main Mosco-convergence results, namely Theorem \ref{thm:Mosco in CDKN} and Theorem \ref{thm:Mosco in MCP}. 
\subsection{The infinite-dimensional setting}
We start with a result dealing with more generally infinite-dimensional spaces. The conclusion (i) in Theorem \ref{thm:Mosco in CDKN} will directly follow from it.
\begin{theorem}\label{thm:Mosco in CDKinf}
Let $p,q \in (1,\infty)$ be two fixed H\"older conjugate exponents. Let $\Xdmxn$  be a pointed metric measure space satisfying strong ${\sf  CD}_q(K,\infty)$ for some $K\in\R$. Assume that $\X_n \overset{pmG}{\rightarrow} \X_\infty$ for some $\Xdmxinf$.

If $f_n\in L^p(\mm_n)$ converges $L^p$-weak to $f_\infty \in L^p(\mm_\infty)$, then it holds
\[
 \rmCh_p(f_\infty)\le \liminf_{n\to\infty}\rmCh_p(f_n).
\]
If $K\ge0$, the same conclusion holds if $\Xdmxn$ are only assumed $\CD_q(K,\infty)$ spaces.
\end{theorem}
\begin{proof}
If $\liminf_{n\uparrow\infty}\rmCh_p(f_n)=\infty$ there is nothing to prove. So, let us assume without loss of generality that $C\coloneqq \liminf_{n\uparrow\infty} \rmCh_p^{1/p}(f_n)<\infty$ and, up to pass to a not relabeled subsequence, that $C= \lim_{n\uparrow\infty} \rmCh_p^{1/p}(f_n)$ and $f_n \in W^{1,p}(\X_n)$ for all $n\in\N$. In view of the equivalence formulation given by Proposition \ref{thm:Sobolevintegrated}, and recalling Remark \ref{rmk:planboundedW1p}, the proof will be concluded if we show
\begin{equation}
\int f_\infty(\gamma_1)-f_\infty(\gamma_0)\,\d\eta \le \Comp(\eta)^{1/p}\rmKe_q^{1/q}(\eta)C,\label{eq:claim}
\end{equation}
for all $q$-test plans $\eta \in \PP(C([0,1],\X_\infty))$ with bounded support. Indeed, this implies the sought conclusion, i.e.\ $f_\infty \in W^{1,p}(\X_\infty)$ and $\rmCh_p(f_\infty)\le C^p.$

Let us then fix a $q$-test plan $\eta \in \PP(C([0,1],\X_\infty))$ with bounded support and consider by Proposition \ref{prop:q_polygonal} the $M$-polygonal geodesic $\pi^M_{n} \in \PP(C([0,1],\X_{n}))$ for all $M\ge M_0, n\ge n_0(M)$ satisfying all the listed properties. We claim that there are sequences $M_k\uparrow\infty$, and $n_k\ge n_0(M_k)\uparrow\infty$ so that
\begin{itemize}
    \item[{\rm (a)}] $\rmKe_q(\pi_{n_k}^{M_k}) \le \rmKe_q(\eta) +\frac{2}{k}$ for all $k\in\N$; 
    \item[{\rm (b)}] $\Comp(\pi_{n_k}^{M_k}) \le \Comp(\eta) +\frac{2}{k}$ for all $k\in\N$;
    \item[{\rm (c)}]it holds as $k\uparrow\infty $ that
    \[
       \frac{\d (\e_0)_\sharp \pi_{n_k}^{M_k} }{\d \mm_{n_k}} \to \frac{\d (\e_0)_\sharp \eta }{\d \mm_{\infty}},\qquad \frac{\d (\e_
       1)_\sharp \pi_{n_k}^{M_k} }{\d \mm_{n_k}}\to \frac{\d (\e_1)_\sharp \eta }{\d \mm_{\infty}}, \qquad \text{in $L^q$-strong}.
    \]
\end{itemize}
Indeed, by definition of upper limit, and with the same notation as in Proposition \ref{prop:q_polygonal}, we have that for each $k\in \N$ there exists $M_k>M_0$ so that
\begin{align*}
    &\limsup_{n\uparrow\infty} \rmKe_q(\pi_{n}^{M_k}) \le \rmKe_q(\eta) + \frac{1}{k},& &\limsup_{n\uparrow\infty} \Comp(\pi_{n}^{M_k}) \le \Comp(\eta) + \frac{1}{k},\\
    &\limsup_{n\uparrow\infty}\left\| \xi_n - \frac{\d (\e_0)_\sharp \pi^{M_k}_{n}}{\d \mm_n}\right\|_{L^q(\mm_n)} \le \frac{1}{k}, &   &\limsup_{n\uparrow\infty}\left\| \zeta_n - \frac{\d (\e_1)_\sharp \pi^{M_k}_{n}}{\d \mm_n} \right\|_{L^q(\mm_n)}\le \frac 1k.
\end{align*}
By an analogous argument, for each $k\in\N$ there exists $n_k\ge n_0(M_k)$ so that (a),(b) hold and  
\[
\left\| \xi_{n_k} - \frac{\d (\e_0)_\sharp \pi^{M_k}_{n_k}}{\d \mm_{n_k}}\right\|_{L^q(\mm_{n_k})} \le \frac{2}{k},\qquad \left\| \zeta_{n_k} - \frac{\d (\e_1)_\sharp \pi^{M_k}_{n_k}}{\d \mm_{n_k}} \right\|_{L^q(\mm_{n_k})}\le \frac 2k.
\]
Recalling that $\xi_n,\zeta_n$ were converging respectively to $\frac{\d (\e_0)_\sharp \eta }{\d \mm_{\infty}},\frac{\d (\e_1)_\sharp \eta }{\d \mm_{\infty}}$ in $L^q$-strong, we also deduce (c).

We are now ready to conclude the proof. For all $k\in\N$, using Theorem \ref{thm:Sobolevintegrated} we have
\[
\int f_{n_k}(\gamma_1)-f_{n_k}(\gamma_0)\,\d\pi_{n_k}^{M_k} \le \Comp\left(\pi_{n_k}^{M_k} \right)^{1/p}\rmKe_q^{1/q}\left(\pi_{n_k}^{M_k}\right)\rmCh_p^{1/p}(f_{n_k}).
\]
By property (c) and using the convergence of coupling \eqref{eq:pq coupling}, we have 
\begin{equation}
\int f_{n_k}(\gamma_1)-f_{n_k}(\gamma_0)\,\d\pi_{n_k}^{M_k} = 
\int f_{n_k}\,\d (\e_1)_\sharp \pi^{M_k}_{n_k} - \int f_{n_k}\,\d (\e_0)_\sharp \pi_{n_k}^{M_k} \to \int f_\infty(\gamma_1)-f_\infty(\gamma_0)\,\d\eta.
\label{eq:endpoint coupling}
\end{equation}
All in all, using the properties (a),(b), we have
\[
\begin{aligned}
\int f_\infty(\gamma_1)-f_\infty(\gamma_0)\,\d\eta  &=\lim_{k\uparrow\infty} \int f_{n_k}(\gamma_1)-f_{n_k}(\gamma_0)\,\d\pi_{n_k}^{M_k} \\
&\le \limsup_{k\uparrow\infty}
\Comp\left(\pi_{n_k}^{M_k} \right)^{1/p}\rmKe_q^{1/q}\left(\pi_{n_k}^{M_k}\right)\rmCh_p^{1/p}(f_{n_k})\\
&\le \Comp(\eta)^{1/p}\rmKe^{1/q}_q(\eta) \limsup_{k\uparrow\infty} \rmCh_p^{1/p}(f_{n_k}).
\end{aligned}
\]
By the choice of the sequence $(n_k)$, we have $\limsup_{k\to\infty} \rmCh_p^{1/p}(f_{n_k}) = C$ and this proves \eqref{eq:claim}.

Finally, if $K\ge0$, the above argument can be repeated on ${\sf CD}_q(0,\infty)$ spaces with a single interpolation (rather than a polygonal geodesic) by appealing instead to Proposition \ref{prop:q_polygonal K>0}. 
\end{proof}
 Next, we revisit within our framework a lower semicontinuity result on open sets originally obtained for $p=2$ in \cite[Lemma 5.8]{AmbrosioHonda17} in the ${\sf RCD}$ context. The proof is a slight modification of the above arguments, hence we only comment on the main differences.
\begin{proposition}\label{prop:liminf on open}
Let $p,q \in (1,\infty)$ be two fixed H\"older conjugate exponents. Let $\Xdmxn$  be a pointed metric measure space satisfying strong ${\sf  CD}_q(K,\infty)$ for some $K\in\R$. Assume that $\X_n \overset{pmG}{\rightarrow} \X_\infty$ for some $\Xdmxinf$, and let $A\subset \Z$ be any open set, $(\Z,\sfd)$ being the realization of the convergence.

If $f_n\in W^{1,p}(\mm_n)$ converges $L^p$-weak to $f_\infty \in W^{1,p}(\mm_\infty)$, then it holds
\[
     \int_A|Df_\infty|_p^p\,\d \mm_\infty \le \liminf_{n\to\infty}\int_A|Df_n|_p^p\,\d\mm_n.
\]
\end{proposition}
\begin{proof}
    Up to passing to a subsequence, we can assume that the liminf is a limit. Also, being the conclusion monotone for increasing open sets $A$, it is enough to prove the conclusion for $A$ bounded and open. Moreover, by Corollary \ref{cor:identification}, we can alternatively show that
   \[
        C_{f_\infty}(A) \le \lim_{n\to \infty}C_{f_n}(A) \eqcolon C.
   \]
   To this aim, the proof will be concluded if we show
    \begin{equation}
        \int f_\infty(\gamma_1)-f_\infty(\gamma_0)\,\d\eta \le \Comp(\eta)^{1/p}\rmKe_q^{1/q}(\eta)C,\label{eq:claim open}
    \end{equation}
    for all $q$-test plans $\eta \in \PP(C([0,1],\X_\infty))$ satisfying ${\rm dist}([\eta],A^c)>0$ (cf. Corollary \ref{cor:identification} again). 

    Having reduced to the case where $A$ is bounded, we can argue exactly as done in the proof of Theorem \ref{thm:Mosco in CDKinf}, noticing that in the application of Proposition \ref{prop:q_polygonal} we can further increase $M_0,n_0(M)$ so that ${\rm dist}([\pi_n^M],A^c)>0$ for all $M\ge M_0,n\ge n_0$. Thanks to this fact, we can select a joint sequence in $k\in\N$, and by Corollary \ref{cor:identification} write
    \[
        \int f_{n_k}(\gamma_1)-f_{n_k}(\gamma_0)\,\d\pi_{n_k}^{M_k} \le \Comp\left(\pi_{n_k}^{M_k} \right)^{1/p}\rmKe_q^{1/q}\left(\pi_{n_k}^{M_k}\right) C_{f_{n_k}}(A).
    \]
    From here, the conclusion of the proof follows by sending $k\to \infty$ by the very same arguments at the end of the proof in Theorem \ref{thm:Mosco in CDKinf}.
\end{proof}

\subsection{Main results for ${\sf CD}(K,N)$ spaces}
\begin{proof}[Proof of (i) in Theorem \ref{thm:Mosco in CDKN}]
    Notice that, under the present assumption, Theorem \ref{thm:CDq independent} applies. Hence, the sequence $\X_n$ satisfies all the requirements of Theorem \ref{thm:Mosco in CDKinf} for any couple of H\"older exponent $p,q \in (1,\infty)$. The conclusion follows.
\end{proof}
\begin{proof}[Proof of (ii) in Theorem \ref{thm:Mosco in CDKN}]
If $ \liminf_{n\uparrow\infty} |\dD f_n|(\X_n)=\infty$, there is nothing to prove. Thus, we can suppose that $C\coloneqq \liminf_{n\uparrow\infty} |\dD f_n|(\X_n)<\infty$ and, up to pass to a not relabeled subsequence, that $C = \lim_{n\uparrow\infty} |\dD f_n|(\X_n)$ and $f_n \in BV(\X_n)$ for all $n\in\N$.

By the equivalent characterization of BV-functions given by Proposition \ref{thm:BVplans}, it is enough to show
\begin{equation}
\int f_\infty(\gamma_1)-f_\infty(\gamma_0)\,\d\eta \le \Comp(\eta)\Lip(\eta)C, \label{eq:claimBV}
\end{equation}
for every $\infty$-test plan $\eta \in \PP(C([0,1],\X_\infty))$ with bounded support (recall Remark \ref{rmk:BVplanbounded}).

Let us then fix $M\in\N$ and consider the sequence of polygonal geodesics $\infty$-test plans $\pi^M_n \in \PP(C([0,1],\X_n))$ given by Proposition \ref{prop:infty_polgeo}. By the equivalence result of Theorem \ref{thm:BVplans}, we can write
\[
\int f_n(\gamma_1)-f_n(\gamma_0)\,\d\pi_n^M \le \Comp(\pi_n^M)\Lip(\pi_n^M)|\dD f_n|(\X_n),
\]
for all $n\in\N$. Next, we claim that
\[
\begin{split}
\int f_n(\gamma_1)-f_n(\gamma_0)\,\d \pi^M_n \to   \int f_\infty(\gamma_1)-f_\infty(\gamma_0)\,\d \eta.
\end{split}
\]
Indeed, fix any $s>1$ as in \eqref{SobPoincare BV} and consider $s'$ its conjugate exponent. By iv) in Proposition \ref{prop:infty_polgeo}, for all $M\in\N$ $\frac{\d (\e_0)_\sharp \pi^M_n}{\d \mm_n},\frac{\d (\e_1)_\sharp \pi^M_n}{\d \mm_n} $ converge $L^{s'}$-strong respectively to $ \frac{\d (\e_0)_\sharp \eta}{\d \mm_\infty},\frac{\d (\e_1)_\sharp \eta}{\d \mm_\infty}$ as $n\uparrow\infty$. Now, recall that both $\pi^M_n, \eta$ are supported on curves contained in a bounded region, say a ball $B\subset \Z$. We can thus consider $\varphi \in \Lip_{bs}(Z)$ with $\varphi\ge 0$ and $\varphi \equiv 1$ on $B$, and rewrite
\[
    \int f_n(\gamma_1)-f_n(\gamma_0)\,\d \pi^M_n = \int \varphi f_n\, \d (\e_1)_\sharp \pi^M_n -\int \varphi f_n\,\d (\e_0)_\sharp \pi^M_n.
\]
By the Leibniz rule, we also have $\varphi f_n \in BV(\X_n)$ and, thanks to \eqref{SobPoincare BV}, we thus deduce
\[
    \sup_{n\in\N}\|\varphi f_n\|_{L^s(\mm_n)} <+\infty,
\]
using that $C<\infty$ and $\sup_n \|f_n\|_{L^1(\mm_n)} < \infty$. 
In particular, we have that $\varphi f_n$ converges $L^s$-weak to $\varphi f_\infty$. The claim now follows recalling \eqref{eq:pq coupling} and using again that $\varphi f_\infty \circ \e_i = f_\infty \circ \e_i$, for $i=0,1$.

All in all, from the properties (ii),(iii)  of $\pi^M_n$ given by in Proposition \ref{prop:infty_polgeo}, we can estimate
\begin{align*}
\int f_\infty(\gamma_1)-f_\infty(\gamma_0)\,\d \eta &= \lim_{n\uparrow\infty} \int f_n(\gamma_1)-f_n(\gamma_0)\,\d \pi^M_n \\
&\le \limsup_{n\uparrow\infty} \Comp(\pi_n^M)\Lip(\pi_n^M)|\dD f_n|(\X_n) \\
&\le e^{K^- \frac{\Lip^2(\eta)}{M^2}} \Comp(\eta)\Lip(\eta)\limsup_{n\uparrow \infty}|\dD f_n|(\X_n).
\end{align*}
Observe here that, if $K\ge 0$, then the proof is already concluded by choosing $M=1$. In the general case, taking now $M\uparrow \infty$ and recalling that $C=\lim_{n\uparrow \infty}|\dD f_n|(\X_n)$, the claim \eqref{eq:claimBV} follows. 
\end{proof}
\begin{remark}\label{rem:finlser}\rm
    It follows from \cite[Corollary 4.7]{Kell17_2} and \cite{Kell17} that every $N$-dimensional smooth Finsler manifold with Ricci curvature bounded below by $K$ is a $q$-essential nonbranching ${\sf CD}_q(K,N)$ space for all $q\in(1,\infty)$. In this smooth setting, these facts replace the use of Theorem \ref{thm:CDq independent}, thus guaranteeing the validity of all the conclusions in Theorem \ref{thm:Mosco in CDKN}.\fr
\end{remark} 
We next show the analogue of Proposition \ref{prop:liminf on open}, in the BV space under the same assumptions of the main result proved in this section. 
\begin{proposition}\label{prop:liming BV on open}
    Let $\Xdmxn$ with $\mm_n(\X_n)<\infty$ be $q$-essentially non-branching pointed metric measure spaces for all $q \in (1,\infty),n\in\N$. Suppose that $\X_n$ satisfies ${\sf  CD}(K,N)$ for some $K\in\R,N\in[1,\infty)$ and that $\X_n \overset{pmG}{\rightarrow} \X_\infty$ for some $\Xdmxinf$.  Let $A\subset \Z$ be any open set, $(\Z,\sfd)$ being the realization of the convergence.

    If $f_n\in BV(\X_n)$ converges $L^1$-weak to $f_\infty \in BV(\X_\infty)$, then it holds
    \[
         |\dD f_\infty|(A) \le \liminf_{n\to\infty}|\dD f_n|(A).
    \]
\end{proposition}
\begin{proof}
    The proof follows by the very same arguments already presented in the proof (ii) in Theorem \ref{thm:Mosco in CDKN} taking also into account the proof of Proposition \ref{prop:liminf on open}). Here, we rely instead on the identification result given by Theorem \ref{thm:BVplans}, and on the last conclusion of Proposition \ref{prop:infty_polgeo}.
\end{proof} 
\subsection{Main result for ${\sf MCP}(K,N)$ spaces}\label{sec:mcp}
In this part, we push our analysis in a more general setting of essentially non-branching spaces satisfying the measure contraction property to prove our main result Theorem \ref{thm:Mosco in MCP}.

We start by developing compression estimates for Wasserstein geodesics and build suitable polygonal interpolations also in the ${\sf MCP}(K,N)$ class. We follow our previous analysis closely, limiting the details for brevity reasons. For $K\le 0$ and $N \in [1,\infty)$, consider the function $  (0,\infty)\ni r\mapsto \sfC_{K,N}[r] \in (0,\infty)$ as defined by
\begin{align*} 
&\sfC_{K,N}[r] \coloneqq \left(\sup_{\theta \in (0,r)}\sup_{t\in(0,1)} \frac{t}{\tau^{(t)}_{K,N}(\theta)}\right)^N, & &\text{if }K<0, \\
&\sfC_{0,N}(r) \equiv 1,& &\text{if }K=0.
\end{align*}
Notice that its definition depends only on the choice of $K,N$ and $\sfC_{K,N}[r]$ is non-decreasing and so that $\sfC_{K,N}[r]\downarrow 1$ as $r\downarrow 0$. 
\begin{proposition}\label{prop:q compression MCP}
    Let $\Xdm$ be a $q$-essentially non-branching metric measure space for all $q\in(1,\infty)$ satisfying ${\sf MCP}(K,N)$ for some $K \in \R, N \in [1,\infty)$. Then, for every boundedly supported probability densities $\rho_0,\rho_1\in L^\infty(\mm)$, there is a unique $\pi \in \OptGeo_q(\rho_0\mm,\rho_1\mm)$ so that
    \[
        \Comp(\pi) \le 2^N \sfC_{K^-,N}\big[\Lip(\pi)\big] \|\rho_0\|_{L^\infty(\mm)}\vee \|\rho_1\|_{L^\infty(\mm)}, \qquad \forall t\in[0,1].
    \]
    In particular, $\pi$ is a $q$-test plan.
\end{proposition}
\begin{proof}
We can replace $K$ by $K^-$ in this proof. The fact that  $q$-optimal plans are unique for absolutely continuous marginals is well known (see \cite[Theorem 5.8]{Kell17}, recalling that $\mm$ is qualitatively non-degenerate (\cite{CavallettiHuesmann15}) and $q$-essentially non-branching assumption is in place as an assumption). We observe that the unique $\pi \in \OptGeo_q(\rho_0\mm,\rho_1\mm)$ satisfies
    \[
        \rho_t(\gamma_t) \le 2^N \sfC_{K^-,N}\big[\sfd(\gamma_0,\gamma_1)\big] \, \|\rho_0\|_{L^\infty(\mm)}\vee \|\rho_1\|_{L^\infty(\mm)},\qquad \pi\text{-a.e.\ }\gamma \text{ and }\forall t\in[0,1],
    \]
where we denoted $(\e_t)_\sharp \pi \coloneqq \rho_t\mm$. This would conclude the proof showing also that $\pi$ is a $q$-test plan (as it is concentrated on equi-Lipschitz geodesics). The above estimate was already noticed in \cite[Remark 3.8]{NobiliPasqualettoSchultz22} and can be proved by rearranging as in \cite[Proposition 9.1]{CavallettiMilman21} (working with $q$-essentially non-branching) and thanks to the fact that $\mcp$ is independent of $q$.
\end{proof}
As a by-product, we have the existence of many $W_\infty$-geodesics with effective density estimates.
\begin{proposition}\label{prop:good infty plan mcp}
Let us assume that, for a given sequence  $q_k\uparrow \infty$,  $\Xdm$ is a $q_k$ essentially non-branching metric measure space satisfying ${\sf MCP}(K,N)$ for some $K \in \R, N \in [1,\infty)$ and for all $k\in\N$. Then, for every boundedly supported probability densities $\rho_0,\rho_1\in L^\infty(\mm)$, there is $\pi \in \OptGeo_\infty(\rho_0\mm,\rho_1\mm)$ so that
    \[
        \Comp(\pi) \le 2^N\sfC_{K^-,N}\big[W_\infty(\rho_0\mm,\rho_1\mm)\big]\, \|\rho_0\|_{L^\infty(\mm)}\vee \|\rho_1\|_{L^\infty(\mm)}, \qquad \text{for all }t\in[0,1].
    \]
    In particular, $\pi$ is an $\infty$-test plan.
\end{proposition}
\begin{proof}
Without loss of generality, we shall prove the statement with $K^-$. For brevity, we write $\mu_i\coloneqq \rho_i\mm$, for $i=0,1$. For each $k\in\N$, let us consider $\eta_k \in \OptGeo_{q_k}(\mu_0,\mu_1)$ given by Proposition \ref{prop:q compression MCP}. Define $\eps_k \coloneqq W_\infty(\mu_0,\mu_1)({q_k}^{1/q_k}-1)$ and set
    \[
    \Gamma_k\coloneqq \Big\{ \gamma \in C([0,1],\X) \colon \int_0^1 |\dot \gamma_t|^{q_k}\,\d t \le (W_\infty(\mu_0,\mu_1)+\eps_k)^{q_k}\Big\},\qquad \pi_k \coloneqq \frac{\eta_k\restr{\Gamma_k}}{\eta_k(\Gamma_k)},
    \]
whenever $\eta_k(\Gamma_k)>0$. By Chebyshev's, we have 
\[
 \pi_k(\Gamma^c_k)\le \frac{W^{q_k}_{q_k}(\mu_0,\mu_1)}{q_kW^{q_k}_{\infty}(\mu_0,\mu_1)} \le \frac1{q_k}\to 0,\qquad \text{ as }k\uparrow \infty,
\]
so that,  eventually, $ \pi_k$ is well-defined for all $k$ up to pass to a subsequence. In particular, it holds
\[
 (\e_i)_\sharp \pi_k  \le \frac{(\e_i)_\sharp \eta_k}{\eta(\Gamma_k)} 
  \le \frac 1{\eta(\Gamma_k)}\|\rho_0\|_{L^\infty(\mm)}\vee \|\rho_1\|_{L^\infty(\mm)}\mm ,\qquad \forall i=0,1.
\]
Being a restriction, $\pi_k$ is optimal with respect to its own marginals (cf. Lemma \ref{lem:subplan is optimal}) and, by Proposition \ref{prop:q compression MCP}, it then must hold
\begin{align*}
   \Comp(\pi_k)&\le \frac{2^N}{\eta(\Gamma_k)}\sfC_{K^-,N}\big[ \Lip(\pi_k)\big]\|\rho_0\|_{L^\infty(\mm)}\vee \|\rho_1\|_{L^\infty(\mm)} \\
   &\le \frac{2^N}{\eta(\Gamma_k)}\sfC_{K^-,N}\big[ W_\infty(\mu_0,\mu_1)+\eps_k
 \big]\|\rho_0\|_{L^\infty(\mm)}\vee \|\rho_1\|_{L^\infty(\mm)},
\end{align*}
having used that $\Lip(\pi_k) \le W_\infty(\mu_0,\mu_1)+\eps_k$ by construction and since $r\mapsto \sfC_{K^-,N}[r]$ is non-decreasing. Hence, we have $\sup_k \Comp(\pi_k)<\infty$, $\sup_k\Lip(\pi_k)<\infty$ and, therefore, we have that $(\pi_k)$ is a sequence of $\infty$-test plans satisfying all the assumptions of Proposition \ref{prop:infty_tightness}. Up to a further subsequence, $\pi_k \weakto  \pi$ for some $\infty$-test plan  $ \pi$. Finally, the verification that $\pi$ does the job goes exactly as in the last part of the proof of Proposition \ref{prop:good infty plan}.
\end{proof}
We are finally ready to prove Theorem \ref{thm:Mosco in MCP}. We shall only highlight the necessary modifications from the proof of Theorem \ref{thm:Mosco in CDKN}. We split the proof into two parts, one for each conclusion.
\begin{proof}[Proof of (i) of Theorem \ref{thm:Mosco in MCP}]
Let us fix $p,q\in(1,\infty)$ two arbitrary H\"older conjugate exponents. The argument is analogous to that of Theorem \ref{thm:Mosco in MCP} (in fact, Theorem \ref{thm:Mosco in CDKinf}) using the Lagrangian characterization of Sobolev functions given by Theorem \ref{thm:Sobolevintegrated}. Thus, here we only prove that, for a given $q$-test plan $\eta \in \PP(C([0,1],\X_\infty))$ with bounded support, the very same polygonal construction $\pi_{n}^M \in \PP(C([0,1],\X_n))$ deduced in Proposition \ref{prop:q_polygonal} can be performed with exactly the same properties (i),(ii),(iv) and (iii) replaced by
\begin{itemize}
    \item[(iii')] $\limsup_{M\uparrow\infty}\limsup_{n\uparrow \infty} \Comp(\pi^M_{n}) \le 2^N\Comp(\eta).$
\end{itemize}
As said, the existence of such a polygonal approximation will directly give the conclusion by repeating verbatim the proof of Theorem \ref{thm:Mosco in CDKinf}. 

So, arguing exactly as in Proposition \ref{prop:q_polygonal} we get (with the same notation) for $M\ge M_0$ and $n\ge n_0\coloneqq n_0(M)$,  optimal dynamical plans $ (\eta_{i,n})_{i=0}^{M-1}\subset \PP(C([0,1],\X_n))$ satisfying
\[
\Comp(\eta_{i,n})  \le 2^N\sfC_{K^-,N}\Big[  W_q(\tilde{\rho}_{i,n}\mm_n,\tilde{\rho}_{i+1,n}\mm_n )^{\frac{q-1}{2q}}\Big]  \|\rho_{i,n}\|_{L^\infty(\mm_n)}\vee \|\rho_{i+1,n}\|_{L^\infty(\mm_n)},
\]
for all $i=0,...,M-1$ and $n\ge n_0$, having used that $r \mapsto \sfC_{K^-,N}[r]$ is non-decreasing and the very same construction of $\eta_{i,n}$ as in Proposition \ref{prop:q_polygonal} (relying on Proposition \ref{prop:q compression MCP} that replace the role of Theorem \ref{thm:rajalaCDqKN}). Now, with a usual gluing argument, for all $M\ge M_0$ and $n\ge n_0$  we can build an $M$-polygonal geodesic $\pi^M_{n} \in \PP(C([0,1],\X_{n}))$ satisfying
\[
\Comp(\pi^M_{n_k} )  \le 2^N\sfC_{K^-,N}\left[  \bigvee_{i=0}^{M-1}W_q(\tilde{\rho}_{i,n}\mm_n,\tilde{\rho}_{i+1,n}\mm_n )^{\frac{q-1}{2q}}\right]  \bigvee_{i=0}^M\|\rho_{i,n}\|_{L^\infty(\mm_n)}.
\]
From the above, the property (iii') follows by the very same justifications of Proposition \ref{prop:q_polygonal}, taking into account that $\sfC_{K^-,N}[r]\downarrow 1$ as $r\down 0$.
\end{proof}

\begin{proof}[Proof of (ii) in Theorem \ref{thm:Mosco in MCP}]
    The argument is analogous to that of Theorem \ref{thm:Mosco in CDKN} using the Lagrangian characterization of BV-functions given by Theorem \ref{thm:BVplans}. 
    
Thus, here we only prove that, for a given $\infty$-test plan $\eta \in \PP(C([0,1],\X_\infty))$ with bounded support, the very same polygonal construction $\pi_n^M \in \PP(C([0,1],\X_n))$ deduced in Proposition \ref{prop:infty_polgeo} can be performed with exactly the same properties (i),(ii),(iv) and (iii) replaced by
\begin{itemize}
    \item[(iii')] $\limsup_{n\uparrow \infty} \Comp(\pi^M_n) \le 2^N\sfC_{K^-,N}\Big[\frac{\Lip(\eta)}{M}\Big]\Comp(\eta),$ \, for all $M\in\N$.
\end{itemize}
As said, the existence of such a polygonal approximation will directly give us the conclusion by repeating verbatim the proof of (ii) in Theorem \ref{thm:Mosco in CDKN}.

So, arguing exactly as in Proposition \ref{prop:infty_polgeo} and (with the same notation) for $M\ge M_0$ and $n\ge n_0\coloneqq n_0(M)$ we can invoke Proposition \ref{prop:good infty plan mcp} (replacing the role of Proposition \ref{prop:good infty plan}) to find plans $\eta_{i,n} \in \OptGeo_\infty(\rho_{i,n}\mm_n,\rho_{i+1,n}\mm_n)$ with bounded supports satisfying
\[
\Comp(\eta_{i,n}) \le 2^N\sfC_{K^-,N}\big[ W_{\infty}(\rho_{i,n}\mm_n,\rho_{i+1,n}\mm_n)\big] \|\rho_{i,n}\|_{L^\infty(\mm_n)} \vee \|\rho_{i+1,n}\|_{L^\infty(\mm_n)},
\]
for every $i=0,...,M_1$, having used that $r\mapsto  \sfC_{K^-,N}[r]$ is decreasing. Now, with a gluing argument, we get $(\pi^M_n)\subset \PP(C([0,1],\X_n))$ satisfying, by construction, 
\begin{equation}
   \Comp(\pi^M_n) \le 2^N \sfC_{K^-,N}\left[\bigvee_{i=0}^{M-1}   W_{\infty}(\rho_{i,n}\mm_n,\rho_{i+1,n}\mm_n)\right]\bigvee_{i=0}^M\|\rho_{i,n}\|_{L^\infty(\mm_n)}, \label{eq:comp polyg MCP}
\end{equation}
From the above, the property (iii') follows by the very same justifications of Proposition \ref{prop:infty_polgeo}, taking into account that $\sfC_{K^-,N}[r]\downarrow 1$ as $r\down 0$.
\end{proof}
\begin{remark}[Comparison with \cite{GigliNobili22}]\label{rem:polgeo vs old polgeo}
\rm 
Recall the following equivalent Lagrangian characterization of $W^{1,p}(\X)$: there is $G \in L^p(\mm)$ called $p$-weak upper gradient so that
\begin{equation}
\int f(\gamma_1)-f(\gamma_0)\,\d\eta \le \iint_0^1 G(\gamma_t)|\dot \gamma_t|\,\d t\d \eta,
\label{eq:sob class double integral}
\end{equation}
for all $q$-test plans $\eta \in \PP(C([0,1],\X)$, where $p,q\in(1,\infty)$ are H\"older conjugate. It turns out that $|Df|_p$ can be equivalently characterized as the minimal $L^p$-normed element of the $G$'s satisfying the above. This Lagrangian notion is not standard, but equivalent to those originally proposed in \cite{Shanmugalingam00} and later developed in \cite{AmbrosioGigliSavare11,AmbrosioGigliSavare11-3} (we refer to \cite{AmbrosioIkonenLucicPasqualetto24} for details and relevant references).

In \cite{GigliNobili22}, we also dealt with polygonal interpolations in \eqref{eq:sob class double integral} with the key advantage that the right-hand side is linear in $\eta$. In the varying setting $\Xdmxn$, this seems not suitable to pass to the limit under polygonal approximations and with the $p$-weak upper gradient $G_n=|Df_n|_p$ also varying. Notice, however, that the condition \eqref{eq:sob class double integral} is \emph{stronger} than the ``integrated'' version obtained by H\"older inequality
\[
\iint_0^1 |Df|_p(\gamma_t)|\dot \gamma_t|\,\d t\d \eta \overset{\eqref{eq:def_test_plan1}}{\le} \Comp(\eta)^{1/p}\rmKe^{1/q}_q(\eta)\rmCh^{1/p}_p(f).
\]
Fortunately, this type of control still detects the space $W^{1,p}(\X)$ as proved in Theorem \ref{thm:Sobolevintegrated} and completely decouples the duality of the $p$-Cheeger and $q$-Kinetic energies, provided one can deal quantitatively with compression estimates of Wasserstein geodesics under Ricci lower bounds. 
\fr 
\end{remark}

\section{Applications to Neumann eigenvalues}
In this section, we prove our main result concerning the continuity of Neumann eigenvalues.
\begin{proof}[Proof of Theorem \ref{thm:neumann}]
    For every $f_\infty \in W^{1,p}(\X_\infty)$ with $\int f_\infty|f_\infty|^{p-2}\,\d \mm_\infty =0$, we can take a recovery sequence $f_n \in W^{1,p}(\X_n)$ that is $L^p$-strongly converging to $f_\infty$ satisfying $\limsup_{n\to\infty}{\rm Ch}_p(f_n) \le \rmCh_p(f_\infty)$ (see Theorem \ref{th:limsup}). By \cite[Lemma 9.2]{AmbrosioHonda17}, we have $\|f_n-a_n\|_{L^p(\mm_n)} \to \|f_\infty\|_{L^p(\mm_\infty)}$, where $a_n\in \R$ is the unique number such that $\int (f_n-a_n)|f_n-a_n|^{p-2} \, \d\mm_n=0$. It then holds
    \[
         \frac{\rmCh_p(f_\infty)}{\|f_\infty\|^p_{L^p(\mm_\infty)}} \ge \limsup_{n\to\infty}  \frac{\rmCh_p(f_n)}{\|f_n-a_n\|^p_{L^p(\mm_n)}} \ge \limsup_{n\to\infty} \lambda_p(\X_ n),
    \]
    and hence by arbitrariness of $f_\infty$ we obtain $\lambda_p(\X_\infty) \ge \limsup_n \lambda_p(\X_n)$. \\
    We need to show the converse inequality to conclude the proof. For every $\eps>0,n\in\N$, let $f_n \in W^{1,p}(\X_n)$ be such that
    \[
        \lambda_p(\X_n) \ge \frac{\rmCh_p(f_n)}{\|f_n\|^p_{L^p(\mm_n)}} -\eps,\qquad \int f_n|f_n|^{p-2}\,\d\mm_n=0.
    \]
    Set $\tilde f_n \coloneqq f_n/\|f_n\|_{L^p(\mm_n)} $ so that $\sup_{n\in\N} \|\tilde f_n\|_{W^{1,p}(\X_n)}<\infty$. Since $\X_n,\X_\infty$ are all uniformly bounded by assumptions, have uniform doubling and Poincar\'e constants (\cite{Sturm06I,Sturm06II,Rajala12-2}),  and since boundaries of metric balls are negligible thanks to Bishop-Gromov monotonicity, we can invoke the compactness theorem in \cite[Theorem 6.14]{Wu25} to deduce that, up to a non relabelled subsequence, we have that $\tilde f_n$ converges $L^p$-strong to some function $\tilde f_\infty \in W^{1,p}(\X_\infty)$ with $\|\tilde f_\infty\|_{L^p(\mm_\infty)}=1$. All in all, we deduce again by \cite[Lemma 9.2]{AmbrosioHonda17} that $\int \tilde f_\infty |\tilde f_\infty|^{p-2} \,\d\mm_\infty=0$ and finally that
    \[
       \eps +  \liminf_{n\to\infty} \lambda_p(\X_n) \ge  \liminf_{n\to\infty}\rmCh_p(\tilde f_n)  \ge \rmCh_p(\tilde f_\infty) \ge \lambda_p(\X_\infty),
    \]
    having used Theorem \ref{thm:Mosco in CDKN} in the second inequality. By arbitrariness of $\eps>0$, the proof is concluded.
\end{proof}

\medskip

\noindent\textbf{Acknowledgements}. F.N. and F.V. are members of INDAM-GNAMPA. F.N  acknowledges support by the European Union (ERC, ConFine, 101078057), the INdAM-GNAMPA Project ``Analisi e Gamma-convergenza per alcuni funzionali non locali'' CUP E53C25002010001\#, and the MIUR Excellence Department Project
awarded to the Department of Mathematics, University of Pisa, CUP I57G22000700001.
F.R and F.V. acknowledge support by the MIUR-PRIN 202244A7YL project ``Gradient Flows and Non-Smooth Geometric Structures with Applications to Optimization and Machine Learning". Part of this work was carried out while F.N. was a visitor at the The Erwin Schrödinger International Institute for Mathematics and Physics during the Thematic Program on Free Boundary Problems, Wien 2025. The warm hospitality and the stimulating atmosphere are gratefully acknowledged. F.N. thanks C. Brena and T. Rajala for useful comments, and N. Gigli for stimulating discussions regarding Remark \ref{rem:polgeo vs old polgeo}. The authors thank L. Ambrosio for his interest and the useful suggestions, and S. Honda for suggesting to investigate Theorem \ref{thm:neumann}.

\end{document}